\documentclass[10pt,reqno]{amsart}
\usepackage[letterpaper,margin=1in]{geometry}
\usepackage{tikz}
\usepackage{comment}
\usetikzlibrary{quotes}
\usetikzlibrary{arrows}
\usepackage{amsthm,amssymb,latexsym}
\usepackage{amsmath}
\usepackage[english]{babel}
\usepackage[utf8]{inputenc}
\usepackage{amsthm}
\usepackage{array}
\usepackage{framed}
\usepackage{pgfplots}
\usepackage{textcomp}
\usepackage{float}
\usepackage{graphicx}
\usepackage{caption}
\usepackage{subcaption}
\usepackage{upgreek}

\usepackage{algpseudocode}
\usepackage{algorithm}

\usepackage{amssymb}
\usepackage{thmtools}
\usepackage{enumerate}
\usepackage{bm}

\usepackage{array, makecell} 

\usepgfplotslibrary{fillbetween}
\pgfplotsset{compat=1.17}

\usepackage{mathtools}

\usepackage{cite}

\usepackage[colorlinks=true,backref=page,bookmarksopen=true]{hyperref} 
\usepackage[capitalise]{cleveref} 
\usepackage[font=small]{caption}

\numberwithin{equation}{section}     

\declaretheorem[numberwithin=section]{theorem}  
\declaretheorem[sibling=theorem]{corollary} 
\declaretheorem[sibling=theorem]{lemma}

\declaretheorem[sibling=theorem, style=remark]{remark}

\usepackage{xcolor}
\hypersetup{
    colorlinks,
    linkcolor={red!50!black},
    citecolor={blue!50!black},
    urlcolor={blue!80!black}
}

\renewcommand*{\backref}[1]{}
\renewcommand*{\backrefalt}[4]{\tiny 
  \ifcase #1 (\textbf{NOT CITED.})%
  \or    (Cited on page~#2.)%
  \else   (Cited on pages~#2.)%
  \fi}

\newcommand{\vertiii}[1]{{\left\vert\kern-0.25ex\left\vert\kern-0.25ex\left\vert #1 
    \right\vert\kern-0.25ex\right\vert\kern-0.25ex\right\vert}}

\newcommand{\CC}{\mathbb{C}}
\newcommand{\RR}{\mathbb{R}}
\newcommand{\NN}{\mathbb{N}}

\newcommand{\bu}{\bm{u}}
\newcommand{\bw}{\bm{w}}
\newcommand{\bx}{\bm{x}}
\newcommand{\pa}{\partial}
\newcommand{\bnu}{\boldsymbol{\nu}}

\newcommand{\dive}{\text{\normalfont div\,}}

\definecolor{Green}{rgb}{0.010,0.7,0.02}

\newcommand{\beq}{\begin{equation}}
\newcommand{\eeq}{\end{equation}}

\newcommand{\bal}{\begin{align}}
\newcommand{\eal}{\end{align}}

\newcommand{\bv}{\bm{v}}
\newcommand{\ba}{\bm{a}}
\newcommand{\by}{\bm{y}}

\newcommand{\bn}{\bm{n}}
\newcommand{\bz}{\bm{z}}

\newcommand{\cF}{\mathcal{F}}

\newcommand{\ep}{\epsilon}

\newcommand{\bell}{\boldsymbol{\ell}}

\newcommand{\ccdot}{\boldsymbol{\cdot}}
\newcommand{\vvdot}{\bm{:}}

\newcommand{\half}{\frac{1}{2}}
\newcommand{\nablahat}{\widehat{\nabla}}
\newcommand{\bg}{\bm{g}}
\newcommand{\bh}{\bm{h}}
\newcommand{\bgtilde}{\widetilde{\bg}}
\newcommand{\bhtilde}{\widetilde{\bh}}
\newcommand{\buBar}{\overline{\bu}}
\newcommand{\buBarDot}{\Dot{\buBar}}
\newcommand{\buDot}{\Dot{\bu}}

\newcommand{\cJ}{\mathcal{J}}
\newcommand{\cU}{\mathcal{U}}
\newcommand{\BB}{\mathbb{B}}
\newcommand{\MM}{\mathbb{M}}
\newcommand{\bb}{\bm{b}}
\newcommand{\bzero}{\bm{0}}
\newcommand{\bmP}{\bm{P}}

\newcommand{\bF}{\mathbf{F}}
\newcommand{\bt}{\bm{t}}
\newcommand{\tang}{\text{\normalfont tan}}
\newcommand{\btheta}{\boldsymbol{\theta}}

\title[Shape derivative]{A shape derivative algorithm for reconstructing elastic dislocations in geophysics}
\author[A. Aspri {\em et Al.}]{A. Aspri, E. Beretta, A. Lee and A. Mazzucato}
\date{\today}

\begin{document}

\begin{abstract}
We consider the inverse problem of determining an elastic dislocation that models a seismic fault in the quasi-static regime of aseismic, creeping faults, from displacement measurements made at the surface of Earth. We derive both a distributed as well as a boundary shape derivative that encodes the change in a misfit functional between the measured and the computed surface displacement under infinitesimal movements of the dislocation and infinitesimal changes in the slip vector, which gives the displacement jump across the dislocation. We employ the shape derivative in an iterative reconstruction algorithm. We present some numerical test of the reconstruction algorithm in a simplified 2D setting.   
\end{abstract}


\maketitle

\section{Introduction} \label{s:intro}

This work concerns an inverse problem for elastic dislocations that model seismic faults. Specifically, we study an iterative reconstruction algorithm for the location and geometry of the fault and the slippage of the rock along the fault. Our algorithm is based on a shape derivative that encodes the changes of a least-square misfit functional between the measured data and a computed solution  under infinitesimal deformation of the fault and infinitesimal changes in the slip along the fault.

As customary in seismology, we model the Earth's Crust as a linearly elastic medium, but we allow the medium to be inhomogeneous. \textit{We assume throughout that the stiffness parameters in the interior of the Earth are known.} We work locally in a patch in the Crust, represented by a bounded domain in $\RR^d$, $d=2,3$. The two-dimensional case is a reduced model  that takes into consideration only depth effects in a vertical slice of the Crust. The surface of the Earth is clearly not subject to any elastic load, so homogeneous Neumann conditions can be imposed, while there are different physical choices of boundary conditions in the buried part.
The fault is modeled as an open surface across which the elastic displacement jumps, while the normal stress remains continuous. The jump  in the displacement measures how much the rock has slipped on each side of the fault surface.

The forward problem or direct problem consists in determining the elastic displacement outside of the fault from knowledge of the fault surface and the slip. The inverse problem consists in determining the fault surface and the slip from knowledge of the displacement at the surface of the Earth. The surface displacement can be measured from satellite interferometric data or data from global positioning systems (GPS). (For a general introduction to the geophysical model we refer to \cite{Eshelby73,Segall10} and references therein.)

While the study of elastic dislocations is by now classic, starting with the seminal work of Volterra and Somigliana at the turn of the last century (see \cite{V1907} and the historical perspective in \cite{F84}), and elastic dislocations have been used to model creeping tectonic faults during the interseismic period, especially at subduction zones in the Earth's Crust ( among the vast literature, we mention \cite{seismic1,seismic2,seismic3}), until recently most of the available results in the literature dealt with an idealized setting, in which the surface of the Earth is modeled as an infinite half-space (or plane) with constant elasticity parameters, the fault has the simple geometry of a rectangle (or segment) with either horizontal or oblique orientation, and the slip is constant.
In this specialized case then Okada obtained an explicit formula for the elastic displacement \cite{Okada92},  using the elastic Green's function in a half-space derived by Mindlin \cite{Mindl54}.

The forward problem can be recast as a singular source problem or as a non-homogeneous transmission problem that can be addressed  using variational techniques or potential theory depending on the regularity of the coefficients, the data, and the domain, including the fault. This has led to establishing well-posedness  under rather general assumptions on the elastic parameters, the geometry of the fault, and the slip
\cite{ABMdH19,ABM23,Vol09,Zwieten_et_al13,Zwieten_et_al14}.

The inverse problem consists of determining uniquely and then reconstructing the fault surface and the slip from remote measurements. In this work, we consider measurements of the elastic displacement made on an open patch of the Earth's surface. Uniqueness in the inverse problem has already been established under certain geometric conditions on the fault surface and the slip \cite{Vol09,ABMdH19} (see also the recent preprint \cite{DLMarXiv23}). Essentially, the fault surface must contain at least one corner singularity.
When the slip does not vanish at the boundary of the fault surface, singularities appear in the elastic displacement near such boundary, which can help in the reconstruction even in the presence of noise. However, these singularities are nonphysical, as the elastic displacement diverges. If the slip is assumed to vanish at the boundary of the fault, a physically justifiable assumption, given that typically only a part of the fault is active at a time, then in fact uniqueness holds under less stringent and more natural geometric assumptions on the fault \cite{ABM23}. Informally, faults have to be graphs but can be arbitrarily oriented. 

We recall the assumptions for both the direct as well as the inverse problem in more detail in Section \ref{s:prelim}.

The reconstruction algorithms available in the literature (we mention in particular the works \cite{IV06,IV09,VVI17, Vol22} and references therein) consider the case of a homogeneous, linearly elastic half-space and a planar fault. In this case, there is a layer potential representation for the solution with an explicit kernel. In the case of variable coefficients and more general fault geometries, the layer potential approach is more difficult to use. The approach we take instead is to minimize a misfit functional that measures the error between the computed solution and the measurements at the Earth's surface, using a distributed shape derivative. The shape derivative is obtained by solving a certain adjoint problem, which can be done efficiently using for instance a Finite Element Method (FEM) or a Discontinuous Galerkin (DG) Method. shape optimization is a well developed tool for geometric identification and reconstruction (we refer the reader, for example, to the monograph \cite{Zolesio92}). The regularity of the boundary and interfaces has also been carefully studied, at least for scalar problems. We mention in particular the work \cite{Laurain20}.

The minimization of the misfit functional follows a steepest descent algorithm, which we implement in practice only in the two-dimensional case, borrowing the approach developed in \cite{BMPS18} for reconstruction of polygonal partitions in EIT. In this approach, the vertices of the dislocation line are assumed to be nodes of a coarse domain partition and are moved individually using the shape derivative calculated along linear elements on the partition. The optimal position of each vertex is computed separately and used to update the polygonal surface. The advantage of this method is that it is local, hence it can be parallelized to increase efficiency. The drawback is its accuracy. To calculate the shape derivative, we employ a DG method to compute both the forward as well as the adjoint problem on a finer conforming mesh.  The steepest descent algorithm is stopped when a given tolerance is reached. As with most iterative methods, we have a good reconstruction only if the initial guess is sufficiently close to the true data for the problem (the fault surface and the slip vector). Furthermore, the limited accuracy of our approach, the inherent sensitivity of the shape derivative to small changes,  and the underlying ill-posedness of the inverse problem makes the reconstruction algorithm less stable. We find, consequently that we have to make more than one measurement or measure on a sufficiently large portion of the boundary.

In this work, we focus on the reconstruction of the fault surface assuming the slip is known (though we include one numerical test for the reconstruction of the slip). Indeed, the inverse problem is linear in the slip vector and hence stable, while the inverse problem for the fault surface is nonlinear.
We do not address the stability of the reconstruction, nor the convergence of the method, which we reserve to address in future work. 
In the case of planar faults and homogeneous elastic parameters,  Lipschitz stability for the inverse fault problem was proved in \cite{TV19}, albeit leading to a non-quantitative estimate, which therefore cannot be directly employed to establish convergence. Quantitative Lipschitz stability estimates for a linear crack were obtained in dimension 2 in \cite{BFV2008}.

As proof of concept, we have performed numerical experiments in simplified geometries, specifically for linear faults in a two-dimensional region. As expected, we find that when the slip vector is constant, hence singularities appear in the solution at the boundary of the dislocation, the reconstruction is more accurate, while the error is significantly larger if the slip vector is allowed to vanish at the endpoints
(see \cite{BFKL2010} where a reconstruction algorithm for linear cracks in dimension two is presented and tested).

The paper is organized as follows. In Section \ref{s:prelim}, we discuss the set-up for both the forward as well as the inverse problem and recall the uniqueness results for the inverse problem in \cite{ABMdH19,ABM23}. Section \ref{s:shape} is devoted to the rigorous derivation of the shape derivative. Lastly, in Section \ref{s:numerics}, we discuss the numerical tests and their outcomes.

\section{Preliminaries} \label{s:prelim}

We start by discussing the set-up for the forward problem and the assumptions we make on the fault surface and on the slip vector. We follow the notation and conventions in \cite{ABMdH19,ABM23}. We only discuss the three-dimensional case with obvious modification in the two-dimensional case.

We start by setting some notation. We employ the blackboard font e.g.  $\mathbb{A}$ for tensors, the boldface font e.g. $\mathbf{A}$ matrices, bold italics font e.g.  $\bm{a}$ vectors, and italics font for scalars e.g. $a$. We denote matrix-vector multiplication simply by $\mathbf{A}\, \bm{b}$, while $\ccdot$ and $\vvdot$ represents the inner product between vectors and matrices respectively. Lastly, we use Einstein notation of summation over repeated indices, e.g.
\begin{equation*}
   \bm{c}^i = [\mathbf{A} \bm{b}]^{i} = \mathbf{A}^{ij} \bm{b}_j.
\end{equation*}

As in \cite{ABM23}, we are concerned with faults buried in the Earth's Crust that undergo a creeping aseismic movement. Such a situation is typical of areas of microseismicity, but can lead to earthquakes of sizeable magnitude  \cite{CreepFault}. In this context, faults are typically shallow, and it is reasonable to work in a patch of the Earth's Crust, modeled as a bounded domain $\Omega$ with Lipschitz boundary. In the numerical implementation, $\Omega$ is of polygonal type.
The boundary of $\Omega$, $\pa\Omega$, is subdivided into two parts, one that represent the buried part $\Sigma$ and the other, $\pa\Omega\setminus \Sigma$, that lies on the surface of the Earth. We assume that $\Sigma$ is a closed subset of $\pa\Omega$. We model the fault by an open, orientable, Lipschitz surface $S$, with Lipschitz boundary $\pa S$ at positive distance $\delta$ from the boundary of $\Omega$. We assume that $S$ can be extended to a closed Lipschitz surface $\Gamma$ also at positive distance from $\pa\Omega$. Then $\Gamma$ splits $\Omega$ into two (possibly disconnected) components $\Omega^+$ and $\Omega^-$, with $\pa\Omega^-=\Gamma$ and $\Omega^+=\Omega \setminus \overline{\Omega^-}$. We denote the unit outer normal to $\pa\Omega$ with $\bnu$ and $\bn$ denotes the unit outer normal to $\Omega^-$, which induces an orientation on $S$.

We follow a standard notation for function spaces, e.g. $W^{s,p}(\Omega)$,
$s\geq 0$, $1\leq p\leq \infty$, denotes $L^p$-based Sobolev spaces and $W^{s,2}(\Omega)=H^s(\Omega)$. The space $H^{1/2}_0(\Omega)$ denotes, as customary, the closure of $C^\infty_c(\Omega)$ in the $H^s$ norm. We also define Sobolev spaces $H^s(\pa\Omega)$ on the boundary of $\Omega$  and $H^s(S)$ on the fault $S$  for $0\leq s<1$, using partitions of unity and a flattening of the boundary, which is justified by the regularity assumptions we have made on $\Omega$ and $S$. We also introduce Sobolev spaces of vector fields with zero trace on $\Sigma$. If $U$ is an open subset of $\Omega$ and $\Sigma \subset \pa U$, we set:
\beq \label{eq:HsDspaceDef}
  H^s_\Sigma(U) := \{ \bu \in H^s(U), \; \bu\lfloor_\Sigma = \bzero\}, \quad s>\half.
\eeq

We will also employ a less standard space on $S$, the space $H^\half_{00}(S)$, which can be defined as the weighted space:
\begin{equation} \label{eq:H1200def}
  H^{\half}_{00}(S):=\Big\{\bu\in H^{\frac{1}{2}}_{0}(S),
    \delta^{-1/2}\bu\in L^2(D)\Big\},
\end{equation}
where $\delta(x) =$dist$(x,\pa S)$ for $x\in S$ and the distance is computed using geodesic length on $S$. The space $H^{\half}_{00}$ has the property that, if $S$ is seen as an open patch of a closed Lipschitz surface $\Gamma$, then extension by zero to $\Gamma$ is a bounded operator from $H^{\half}_{00}(S)$ to $H^{\half}(\Gamma)$. Furthermore, elements of $H^{\half}_{00}(S)$ have trace zero on $\pa S$. We refer to \cite[Section 2]{ABM23} for further discussion on this point.

Finally, if $f$ defined a.e. on $\Omega$ is regular enough to have a non-tangential limit on $S$, we let $[f]_S$ denotes the jump on $S$, defined as:
\begin{equation*}
    [f]_S =f_+-f_-,
\end{equation*}
where $f_\pm$ is the limit taken non tangentially in $\Omega^\pm$, respectively. 

In the regime of creeping faults, where notable slip occurs without earthquakes, a quasi-static approximation is justified. Therefore, we model the Earth's Crust as a linearly elastic but inhomogeneous material with time-independent elastic parameters. The elastic response is encoded in a fourth-order tensor $\CC$, the elasticity tensor, which we assume to be totally symmetric (that is, the material is hyperelastic) and uniformly Lipschitz continuous in $\Omega$:
\begin{equation} \label{eq:LameCoeffLipschitzCond}
   \CC \in W^{1,\infty}(\Omega).
\end{equation}
This regularity assumption is not optimal in the context of geophysical applications, since the Earth's Crust is typically modeled as a layered medium, where material properties can vary abruptly from one layer to the next. However, Lipschitz regularity is needed to compute the shape derivative and for the solvability of the forward problem in some cases. We discuss this point further later in this section. 

We also impose the standard (uniform) strong convexity conditions:
\begin{equation} \label{eq:StrongConvexCond}
	\mathbb{C}(\bm{x})\widehat{\mathbf{A}}:\widehat{\mathbf{A}}\geq \gamma
	|\widehat{\mathbf{A}}|^{2}, \qquad \textrm{a.e in}\,\, \Omega,
\end{equation}
for all symmetric matrices $\mathbf{A}$ for some constant $\gamma>0$ , which guarantees that the energy functional is strictly positive.

The forward or direct fault problem consists of finding the elastic displacement $\bu$ that solves the homogeneous elastostatics system subject to suitable boundary conditions on $\pa\Omega$ and suitable transmission conditions on $S$. On the part of $\pa\Omega$ that lies on the surface of the Earth $\pa\Omega\setminus \Sigma$, it is appropriate to assume that there is no elastic load, hence $\pa\Omega\setminus \Sigma$ is traction-free. The traction is the normal component of the stress $\boldsymbol\upsigma=\CC\widehat{\nabla}\bu$, with $\widehat{\nabla}$ the symmetric part of the gradient, and represents the elastic force acting on surfaces. 

On the buried part $\Sigma$, there are several boundary conditions that can be imposed. For simplicity, we choose to impose homogeneous Dirichlet conditions, that is, $\Sigma$ is displacement-free. This choice is physically motivated by our assumption that the fault is at a positive distance from the boundary, so the boundary is undisturbed by the slip on the fault. For numerical purposes, absorbing boundary conditions could also be considered.

On the fault surface $S$, we assume continuity of the traction, while there is a jump in the displacement, encoding the slippage of the rock on the two sides of the fault, given by a vector field $\bg$ on $S$, the Burger's vector. The slip needs not be tangential to $S$ and can be oriented arbitrarily.

The displacement $\bu$ must, hence, satisfy the following mixed-boundary-value-interface problem:
\beq \label{eq:ForwardProblem}
  \begin{cases}
     \dive(\CC \nablahat \bu) = \bzero, & 	 \textrm{in}\,\,
			\Omega \setminus \overline{S},\\
			(\mathbb{C}\widehat{\nabla}  \bu)\bm{\nu}=\bm{0}, & \textrm{on}\,\, \partial \Omega \setminus \Sigma, \\
			  \bu=\bm{0}, & \textrm{on}\,\, \Sigma,\\
			[  \bu]_{S}=\bm{g}, \quad 
			[(\mathbb{C}\widehat{\nabla}  \bu)\bm{n}]_{S}=\bm{0}.
  \end{cases}
\eeq
The strong convexity condition \eqref{eq:StrongConvexCond} implies strong ellipticity for the elastostatics operator $\dive \CC\nablahat$. Since we are interested in domains of polyhedral type, elliptic regularity does not hold in standard Sobolev spaces and, consequently, problem \eqref{eq:ForwardProblem} must be intended in a weak or very weak sense, depending on the regularity of the slip $\bg$. 

As a matter of fact, if $\bg$ does not vanish at the boundary of $S$, i.e., if $\bg\in H^\half(S)$, but not in $H^\half_{00}(S)$, then singularities develop at $\pa S$, which precludes any variational formulation. Instead, owing to the Lipschitz regularity of the elastic coefficients,  a very weak solution to \eqref{eq:ForwardProblem} of can be constructed by duality, viewing the mixed-boundary-value interface problem as a singular source problem in the whole of $\Omega$ with source in $H^{-3/2-\ep}$, $\ep >0$. For a precise definition of very weak solution and more details on the dual formulation of problem \eqref{eq:ForwardProblem}, we refer to \cite{ABMdH19}. There the domain is the lower half space and weights must be imposed to ensure integrability at infinity. However, essentially the same arguments apply in bounded domains, without the need of weights. It can be shown using a layer potential representation that the solution is in $H^{1-\ep}(\Omega\setminus {\overline{S}})$ for any $\ep>0$. Hence, the homogeneous Dirichlet condition on $\Sigma$ holds in trace sense. We recall the following well-posedness results for very weak solutions (see \cite[Theorem 3.11]{ABMdH19}). 

\begin{theorem} \label{t:VeryWeakSol}
    There exists a very weak solution $\bu \in H^{\half-\ep}(\Omega)\cap 
    H^{1-\ep}_\Sigma(\Omega \setminus {\overline{S}})$ to problem \eqref{eq:ForwardProblem} for any $\bg \in H^\half(S)$ and any $\ep>0$.
\end{theorem}

By contrast, when $\bg\in H^\half_{00}(S)$, extending the slip by zero to $\Gamma$, yields a standard non-homogeneous interface problem that admits a variational solution in $H^1(\Omega\setminus \overline{S})$, which only requires the elastic parameters to be bounded uniformly. We recall the weak formulation that will be the starting point to compute the shape derivative.
Below, $\bgtilde\in H^\half(\Gamma)$ denotes the extension of $\bg\in H^\half_{00}(S)$ by zero. 
We recall the following standard lemma (see e.g. \cite{ABM23}, Lemma 3.3, and references therein):

\begin{lemma}\label{lem: Htilde}
The following identification holds:
\begin{equation} \label{eq:Htilde}
		H^1(\Omega\setminus\overline{S})\cong	\widetilde{H}:=\Big\{f\in L^2(\Omega):\,\, f\lfloor_{\Omega^+}\in
H^1(\Omega^+),\, f\lfloor_{\Omega^-}\in H^1(\Omega^-),
\textup{\textrm{and}}\,\, [f]_{\Gamma\setminus\overline{S}}=0   \Big\}.
\end{equation}
\end{lemma}

Then, $\bu\in H^1_\Sigma(\Omega\setminus \overline{S})$ is a weak solution of 
the mixed-boundary-value-interface problem \eqref{eq:ForwardProblem} if, for all $\bv \in H^1_\Sigma(\Omega)$:
\beq \label{eq:WeakForm}
  a_\Gamma(\bu,\bv):= \int_{\Omega\setminus\Gamma} \CC\nablahat \bu \vvdot\nablahat \bv \, d\bx = 0, 
\eeq
and $[\bu]_{\Gamma}=\bgtilde$. We note that by the symmetry properties of $\CC$, the bilinear form $a$ defined above can be written equivalently as:
\begin{equation*} 
    a_\Gamma(\bu,\bv)= \int_{\Omega\setminus\Gamma} \CC\nabla \bu \vvdot\nabla \bv \, d\bx= \int_{\Omega\setminus\Gamma} \CC\nabla \bv \vvdot\nabla \bu \, d\bx.
\end{equation*}
We refer the reader to \cite{ABM23} for more details and a proof of the following well-posedness result (see \cite[Theorems 3.5 and 3.6]{ABM23}).

\begin{theorem} \label{t:WeakSol}
    There exists a unique weak solution $\bu \in H^1_\Sigma(\Omega\setminus {\overline{S}})$ to problem \eqref{eq:ForwardProblem}, which depends continuously on the slip vector $\bg\in H^\half_{00}(S)$.
\end{theorem}

While assuming that $\bg$ vanishes at $\pa S$ is more physically justified (only a patch on a fault is typically active), having singularities at the boundary of $S$ helps with the inverse problem of recovering the shape and location of $S$. We consider both cases in the numerical experiments.

\begin{remark}
We present a detailed derivation of the shape derivative in the case $\bg\in H^\half_{00}(S)$, since the variational setting is the best suited for numerical applications. However, the boundary shape derivative can also be obtained using a layer potential representation of the elastic displacement $\bu$, which we established in \cite{ABMdH19}, following the the approach used in \cite{BFV2008}. This derivation applies to any slip field $\bg \in H^\half(S)$, including the case of constant slip.
\end{remark}

We close this section by recalling two uniqueness results for the inverse problem of determining the fault $S$ and the slip vector $\bg$ from surface measurements of the displacement $\bu$ on an open subset $\Xi \subset \pa \Omega \setminus \Sigma$, henceforth called the acquisition manifold. The first result is in the setting of Theorem \ref{t:VeryWeakSol}, while the second is in the setting of Theorem \ref{t:WeakSol}. The two results differ in the a priori assumptions imposed on $S$ and on $\bg$, which are more stringent in the first case.

\begin{theorem}[Theorem 5.1 in \cite{ABMdH19}] \label{t:uniq.IP.veryweak}
Let $S_1, \, S_2$ be two piecewise-linear fault surfaces that in addition are graphs with respect to the same coordinate system.  Let $\bm{g}_i$ be bounded tangential fields in $ H^{1/2}(S_i)$ with \textup{supp}$\,
	\bm{g}_i=\overline{S}_i$, for $i=1,2$. Let $\bu_i$, for $i=1,2$, be the unique very weak solution of \eqref{eq:ForwardProblem}  corresponding to $\bm{g}=\bm{g}_i$ and $S=S_i$. 
	 If $\bu_{1\big|_{\Xi}}=\bu_{2\big|
		_{\Xi}}$,  then $S_1=S_2$ and $\bm{g}_1=\bm{g}_2$.
\end{theorem}

The conditions on $S_i$ and $\bg_i$ in the theorem above can be weakened somewhat. 

\begin{theorem}[Theorem 4.1 in \cite{ABM23}] \label{t:uniq.IP.weak}
Let $S_1,S_2$ be two fault surfaces that in addition are graphs with respect to the same coordinate system. Let $\bm{g}_i\in
H^{\frac{1}{2}}_{00}(S_i)$, for $i=1,2$, with \textup{Supp}$\,
	\bm{g}_i=\overline{S}_i$, for $i=1,2$, and let $  \bu_i$, for $i=1,2$, be the  unique weak solution of \eqref{eq:ForwardProblem} corresponding to $\bm{g}=\bm{g}_i$ and $S=S_i$.
	If $\bu_{1\big|_{\Xi}}=\bu_{2\big|
		_{\Xi}}$,  then $S_1=S_2$ and $\bm{g}_1=\bm{g}_2$.
\end{theorem}

We observe that a single measurement on $\Xi$ is sufficient for the unique determination of both the fault and the slip.
The assumption that surfaces are graphs with respect to a fixed, but arbitrary, coordinate system is not too restrictive in the geophysical context, where nearby faults tend to follow the same orientation. For numerical purposes, we are interested in polyhedral surfaces and, hence, both uniqueness theorems apply.

\section{The shape derivative} \label{s:shape}

In the rest of the paper we present an iterative reconstruction algorithm based on the minimization of a suitable misfit functional. For ease of notation, {\em we use $\lesssim$ to denote $\leq C$, where $C$ is a constant that may depend on the data, i.e., the Lipschitz norm of $\CC$ and the strong convexity constant $\gamma$, on $\Omega$, $S$, and $\bg$}, which are fixed in computing the derivative.

Let $\bu_m$ denote the measurement on the acquisition manifold $\Xi$. Assuming no measurement errors, $\bu_m$ agrees with the trace of the solution to the mixed-boundary-value-interface problem \eqref{eq:ForwardProblem} for a certain fault surface $S_0$ and slip vector $\bg_0$. We let $\cJ$ denote least-squares fitting cost functional:
\beq \label{eq:JfunctDef}
  \cJ(S,\bg) := \frac{1}{2} \int_\Xi |\bu_m(\bx) - \bu(\bx)|^2 \, d\sigma(\bx),
\eeq
where $\bu$ solves \eqref{eq:ForwardProblem} with fault surface $S$ and slip vector $\bg$ and $\sigma(\bx)$ denotes surface area. We remark that $\bg$ depends also on $S$ so we have slightly abused notation. However, since the inverse problem is linear in the slip vector if $S$ is fixed, the key step consists of reconstructing the fault surface $S$. Given the uniqueness results for the inverse problem and the well-posedness of the forward problem, $\cJ$ has a unique global minimum at $S=S_0$ and $\bg=\bg_0$.

To find this minimum, we propose a steepest descent algorithm based on a shape derivative, defined as the derivative of $\cJ$ with respect to infinitesimal movements of the fault $S$ and under infinitesimal changes of the slip vector $\bg$. Due to the ill-posedness of the inverse problem, we regularize the algorithm by minimizing only over piecewise polygonal or polyhedral faults.
We do not study the convergence of this algorithm theoretically, in part because we do not know the convexity properties of $\cJ$. Furthermore, there are measurement errors and we compute the solution of the forward problem numerically. We discuss the influence of these two type of errors in the section on the numerical tests.

In this section, we tackle the computation of the shape derivative in terms of a material derivative, which gives the change of the solution $\bu$ of  \eqref{eq:ForwardProblem} under infinitesimal changes of the fault $S$ and the slip $\bg$.
We encode the movement of the fault $S$ in terms of a bi-Lipschitz diffeomorphism $\phi_\tau:\overline{\Omega}\to \overline{\Omega}$, depending on a small parameter $\tau\in[0,1)$ with the following properties:
\begin{enumerate}
 \item $\phi_0=I$, and $\phi_\tau\lfloor_{\pa\Omega}=I$, where $I$ is the identity map;
 \item $\phi_\tau$ is linear in $\tau$.
\end{enumerate}
Since $S$ is at positive distance from $\pa\Omega$, there exists $\delta>0$ such that dist$(\phi_\tau(S),\pa \Omega)>\delta$ for all $\tau\in [0,1)$. Therefore, we can assume without loss of generality that:
\beq \label{eq:FaultDeformationDef}
    \phi_\tau := I +\tau \, \cU,
\eeq
where $\cU:\Omega\to \Omega$ is a Lipschitz map with support in a neighborhood $U$ of $\Gamma$. (We recall that $\Gamma$ is at positive distance from $\pa\Omega$, so $U$ always exists.) Then:
\begin{equation*}
  \frac{d \phi_\tau}{dt} =\cU, \quad D_{\bx} \phi_\tau = I+\tau D_{\bx} \cU.
\end{equation*}

Similarly, we encode the infinitesimal change of the slip $\bg$ in terms of a vector field $\bg_\tau$, given by:
\beq \label{eq:SlipDeformationDef}
  \bg_\tau  :=\bg + \tau \,  \bh,
\eeq
where $\bh \in H^\half_{00}(S)$, and again denote with $\bgtilde_\tau$ its extension by zero to $\Gamma$. The actual change in the slip is given by the composition of $\bg_\tau$ with $\phi_\tau^{-1}$. 

For notational convenience, we set:
\beq \label{eq:GammaTauDef}
     S_\tau:= \phi_\tau(S), \quad \Gamma_\tau:= \phi_\tau (\Gamma),
\eeq
so that $S_0\equiv S$ and $\Gamma_0\equiv \Gamma$.
We recall that we assume that $\CC$ is known in both the forward and the inverse problem, it is defined on the whole of $\overline\Omega$  and globally Lipschitz continuous.

We first obtain a distributed derivative, which can be computed in the framework of weak solutions of \eqref{eq:ForwardProblem}, assuming the slip $\bg \in H^\half_{00}(S)$. We hence work with $\Gamma_\tau$ and the extensions $\bgtilde_\tau$, $\bhtilde$ of the slip vectors $\bg_\tau$, $\bh$  by zero to $\Gamma_\tau$.

For ease of notation, we write \ $ \cJ(\tau) := \cJ(\Gamma_\tau,\bgtilde_\tau)$.
We can now define the {\em shape derivative} as the directional derivative:
\beq \label{eq:ShapeDerDef}
     D_{(\cU,\bhtilde)} \cJ(\Gamma_\tau,\bgtilde_\tau) := \frac{d \cJ(\tau)}{d\tau}\lfloor_{\tau=0}.
\eeq
Next, we let $\bu_\tau\in H^1_\Sigma(\Omega\setminus \overline{S_\tau})$ be the (weak) solution of the perturbed problem, which exists and is  unique by virtue of Theorem \ref{t:uniq.IP.weak}.:
\beq \label{eq:ForwardProblemPerturbed}
  \begin{cases}
     \dive(\CC \nablahat \bu_\tau) = \bzero, & 	 \textrm{in}\,\,
			\Omega \setminus \Gamma_\tau,\\
			(\mathbb{C}\widehat{\nabla}\bm{u_\tau})\bm{\nu}=\bm{0}, & \textrm{on}\,\, \partial \Omega \setminus \Sigma, \\
			  \bu_\tau=\bm{0}, & \textrm{on}\,\, \Sigma,\\
			[  \bu_\tau]_{\Gamma_\tau}=\bm{\bgtilde_\tau}\circ \phi^{-1}_\tau, \quad 
		[(\mathbb{C}\widehat{\nabla}\bu)\bm{n}]_{\Gamma_\tau}=\bm{0},
  \end{cases}
\eeq
or equivalently:
\begin{align}
    &a_{\Gamma_\tau}(\bu_\tau,\bv)= \int_{\Omega\setminus\Gamma_\tau} \CC\nabla_{\by} \bu_\tau \vvdot\nabla_{\by} \bv \, d\by =0, \qquad
    [\bu_\tau]_{\Gamma_\tau}=\bgtilde_\tau \circ \phi^{-1}_\tau, \nonumber
\end{align}
for every $\bv \in H^1_\Sigma(\Omega)$, where $\by=\phi_\tau(x)$.  

To compute the material derivative, we transform the perturbed problem \eqref{eq:ForwardProblemPerturbed} into a problem on the fixed domain $\Omega \setminus \overline{S}$. We then let 
\begin{equation*}
    \Breve{\bu}_\tau(\bx):= \bu_\tau\circ \phi_\tau(\bx), \quad \bx \in \Omega\setminus \overline{S}.
\end{equation*}
By the change of variable formula, which holds in Lipschitz domains, $\Breve{\bu}_\tau\in H^1_\sigma(\Omega\setminus \overline{S})$ satisfies the variational problem:
\beq \label{eq:WeakFormOriginalDomain}
  a_\tau(\Breve{\bu}_\tau,\bw):= \int_{\Omega\setminus \Gamma_\tau} \CC \left(\nabla( \Breve{\bu}_\tau\circ\phi^{-1}_\tau) \right)\vvdot  \left(\nabla( \bw\circ\phi^{-1}_{\tau})  \right)  \, d\by=
  \int_{\Omega\setminus \Gamma} \CC_\tau \nabla_{\bx} \Breve{\bu}_\tau [D\phi_\tau]^{-1}\vvdot  \nabla_{\bx} \bw [D\phi_\tau]^{-1}\, J(\phi_\tau)\, d\bx=0,
\eeq
where $J(\phi_\tau)=\det (D_{\bx}\Phi_\tau)$ is the Jacobian determinant, $\bw(\bx)= \bv(\phi(\bx))$, and 
\beq \label{eq:CCtauDef}
    \CC_\tau(\bx) := \CC(\phi_\tau(\bx)).
\eeq
Above we have again  used the symmetry properties of $\CC$. Additionally, since $\phi_\tau$ induces an automorphism on $H^1_\Sigma(\Omega)$, we have dropped the dependence on $\tau$ on the test function $\bw$. We stress that $\CC_\tau$ has the same symmetries and satisfies the same strong convexity condition, since they both hold pointwise uniformly in $\Omega$, but it is no longer an elasticity tensor.

It will be convenient to switch to a weak formulation with homogeneous transmission conditions, so as to have the same function space for solutions and test functions. To this end, we introduce a lift of the slip vector on $\Gamma$  using a Newtonian potential. Given any vector field $\bm{a}\in H^\half(\Gamma)$, we define its lift $\boldsymbol{\ell}_{\bm{a}}\in H^1(\Omega\setminus \Gamma)$ as any solution to the problem:
\beq \label{eq:liftDef}
 \begin{cases}
   \Delta \bell_{\bm{a}} = 0 \text{in } \Omega\setminus \Gamma, & \\
   [\bell_{\bm{a}}]_\Gamma ={\bm{a}}, \quad [\pa_{\bn} \bell_{\bm{a}}]_\Gamma= \bzero, &\,\\
   \text{Supp\,} \bell_{\bm{a}} \subset \Omega, & \,
 \end{cases}
\eeq
with $\pa_{\bn}=\bn \ccdot \nabla$ the normal derivative.
We can specify the lift uniquely by requiring that $\bell_{\bm{a}} \equiv 0$ on $\Omega \setminus U$, where $U$ is the neighborhood of $\Gamma$, for instance.
We let:
\beq \label{eq:BuBarDef}
    \buBar_\tau(\bx) := \Breve{\bu}_\tau(\bx) -\bell_{\tau \,  \bhtilde} (\bx), \quad \bx \in \Omega\setminus \Gamma.
\eeq
Then $\buBar_\tau \in H^1_\Sigma(\Omega\setminus\overline{S})$ and satisfies:
\beq \label{eq:BuBarWeakForm}
  a_\tau(\buBar_\tau,\bw)=-a_\tau(\bell_{\tau \,  \bhtilde},\bw), \qquad  [\buBar_\tau]_\Gamma =\bgtilde,
\eeq
since $[\Breve{\bu}_\tau]_{\Gamma}= \bgtilde_\tau=\bgtilde + \tau\bhtilde$.
Furthermore:
\begin{equation*}
     [\buBar_\tau-\bu]_{\Gamma}=\bgtilde-\bgtilde= 0 \ \Rightarrow \  \buBar_\tau-\bu \in H^1_\Sigma(\Omega).
\end{equation*}
We now define the {\em material derivative} $\buBarDot\in H^1_\Sigma(\Omega)$ as:
\beq \label{eq:BuBarDotDef}
    \buBarDot := \frac{d \buBar_\tau}{d\tau}\lfloor_{\tau=0} = \lim_{\tau \to 0^+} \frac{\buBar_\tau -\bu}{\tau},
\eeq 
where the convergence is strongly in $H^1_\Sigma(\Omega)$, if the limit exists. 

Before tackling the existence of the material derivative, we discuss some useful properties of $\CC$ and $\phi_\tau$. We recall that $\phi_\tau$ and  $\CC$ are  uniformly Lipschitz on $\Omega$.
Since $\phi_\tau =I +\tau \, \cU$, there exists  $\tau$ small enough, say $\tau\in [0,\tau_0]$, for some $\tau_0\ll 1$,  such that $\phi^{-1}_\tau \approx I-\tau \, \cU$  and 
\beq \label{eq:DphiTauDer}
  \lim_{\tau\to 0^+}\,\,\dfrac{D\phi_\tau^{-1}-I}{\tau}=- D\cU,\qquad \lim_{\tau\to 0^+} \,\,\dfrac{D\phi_\tau-I}{\tau}= D\cU,
\eeq
 in $L^\infty(\Omega)$. Similarly,  $J(\phi_\tau)\approx J(\phi^{-1}_\tau)\approx 1$ also for $\tau\in [0,\tau_0]$, and a straightforward calculation shows that
 \beq \label{eq:JphiTauDer}
   \lim_{\tau\to 0^+} \,\,\dfrac{J(\phi^{-1}_\tau)-1}{\tau}= - \dive \cU,\qquad \lim_{\tau\to 0^+} \,\,\dfrac{J(\phi_\tau)-1}{\tau}=\dive\cU,
 \eeq 
 also in $L^\infty(\Omega)$.
Additionally:
\begin{equation*}
    \|\CC_\tau\|_{W^{1,\infty}(\Omega)} \lesssim \|\cU\|_{W^{1,\infty}(\Omega)} \|\CC\|_{W^{1,\infty}(\Omega)}.
\end{equation*}
and similarly for $\CC\circ \phi_\tau^{-1}$. For ease of notation, we let
\begin{equation*}
     L:=\|\cU\|_{W^{1,\infty}(\Omega)}.
\end{equation*}

We begin with showing that a weak limit exists. We  observe that, since the lift of the jump on $\Gamma$ satisfies a linear problem, $\bell_{\tau \,  \bhtilde}=\tau \bell_{\bhtilde}$. Consequently, $\Dot{\bell_{\bhtilde}} =\bell_{\bhtilde} $. 

\begin{lemma} \label{lem:DotuBarWeakConv}
Let $\bu$ be the  unique weak solution of \eqref{eq:ForwardProblem} with $\bg\in H^\half_{00}(S)$ and let  $\buBar_\tau$ be given in \eqref{eq:BuBarDef}. Then
$\,\,\dfrac{\buBar_{\tau}- \bu}{\tau}$ converges weakly in $H^1_\Sigma(\Omega)$ and strongly in $L^2(\Omega)$ as $\tau \to 0^+$.
\end{lemma}

\begin{remark}
Although convergence will be along subsequences (not relabeled), the whole sequence will converge as we will show uniqueness of the limit later on. 
\end{remark}

\begin{proof}
By Poincar\'e's inequality, it is enough to establish a uniform bound on $\frac{\|\nabla \buBar_\tau-\nabla\bu\|_{L^2(\Omega)}}{\tau}$. 

We recall that $[D\phi_\tau]^{-1}\approx I-\tau D\cU$ for $\tau\in [0,\tau_0]$,
so that
\begin{align}
  \left|\nablahat\overline{\bu}_\tau(\bx) - \nablahat \bu(\bx)\right|^2 
   &\lesssim L^2 \, \left|\widehat{\left(\nabla\overline{\bu}_\tau(\bx) [D\phi_\tau(\bx)]^{-1}\right)} - \widehat{\left(\nabla \bu(\bx) [D\phi_\tau(\bx)]^{-1}\right)} \right|^2. \nonumber
\end{align}
Next, using the strong convexity assumption \eqref{eq:StrongConvexCond} and the symmetry of $\CC$, we have that for a.e. $\; \bx \in \Omega$,
\begin{equation*}
  \left|\widehat{\left(\nabla\overline{\bu}_\tau(\bx) [D\phi_\tau(\bx)]^{-1}\right)} - \widehat{\left(\nabla \bu(\bx) [D\phi_\tau(\bx)]^{-1}\right)}\right|^2  \lesssim C(L,\gamma)\,
   \mathbb{C} \left[ \nabla (\overline{\bu}_\tau -   \bu) (\bx)
 [D\phi_\tau(\bx)]^{-1} \right]\vvdot\left[ \nabla (\overline{\bu}_\tau -\bu) (\bx) [D\phi_\tau (\bx)]^{-1} \right].
\end{equation*}
It then follows from Korn's Inequality that
 \begin{align}
 \,\,\dfrac{1}{\tau^2}\int_{\Omega} \left|\nabla\overline{\bu}_\tau - \nabla \bu\right|^2 \, d\bx
  &\lesssim C(L,\gamma) \int_{\Omega} \,\,\dfrac{1}{\tau^2} \mathbb{C} \left[ \nabla (\overline{\bu}_\tau -\bu) [D\phi_\tau]^{-1} \right]\vvdot \left[ \nabla (\overline{\bu} _\tau -\bu) [D\phi_\tau]^{-1} \right] 
  \, d\bx. \label{eq:L2NormInitialBound}
\end{align}
Next we note that, since  $\,\,\dfrac{\overline{\bu}_\tau -\bu}{\tau} \in H^1_\Sigma(\Omega)$, it can  be used  as a test function in \eqref{eq:WeakForm} for $\bu$ and in \eqref{eq:WeakFormOriginalDomain} for $\Breve{\bu}_\tau$. Hence, for $\tau>0$
\begin{align}
  \frac{1}{\tau} &\int_{\Omega\setminus\Gamma}\CC \,\nabla\bu \vvdot \nabla \left(\,\,\dfrac{\overline{\bu}_\tau -\bu}{\tau} \right) \,d\bx=0 \nonumber \\
  \frac{1}{\tau} &\int_{\Omega\setminus\Gamma} \mathbb{C}_\tau \nabla \Breve{\bu}_\tau [D\phi_\tau]^{-1} \vvdot \left[\nabla \left(\,\,\dfrac{\buBar_\tau -\bu}{\tau} \right) [D\phi_\tau]^{-1} \right] \, J(\phi_\tau)\,d\bx=0 \label{eq:BuBarMinusBuTestFunction}
\end{align}
We then write, using the fact that we can integrate 
$\buBar_\tau-\bu$ equivalently on $\Omega$ or $\Omega\setminus \Gamma$ or $\Omega \setminus \Gamma _\tau$:
\begin{align}
\,\,\dfrac{1}{\tau^2}  \int_{\Omega} \mathbb{C}  & \left[ \nabla (\overline{\bu}_\tau -\bu)  [D\phi_\tau]^{-1}  \right]\vvdot \left[ \nabla (\overline{\bu} _\tau -\bu) [D\phi_\tau]^{-1}  \right] \, d\bx\nonumber \\
 &= - \frac{1}{\tau^2}\int_{\Omega\setminus\Gamma} \CC \left[ \nabla \bu 
 [D\phi_\tau]^{-1} \right]\vvdot \left[ \nabla (\overline{\bu} _\tau -\bu) 
 [D\phi_\tau]^{-1}   \right]  \,  d\bx \nonumber \\
 & \qquad \qquad  + \frac{1}{\tau^2} \int_{\Omega\setminus \Gamma} \CC \left[\nabla \overline{\bu}_\tau  [D\phi_\tau]^{-1}  \right]\vvdot \left[ \nabla (\overline{\bu} _\tau -\bu) [D\phi_\tau]^{-1} \right] \, d\bx
 =: A + B.  \nonumber 
\end{align}
We tackle the terms $A$ and $B$ separately, exploiting \eqref{eq:BuBarMinusBuTestFunction}.
For the first term,  we have:
\begin{align}
 A&=  - \int_{\Omega\setminus\Gamma} \CC  \left[\nabla \bu \,\,\left(\dfrac{[D\phi_\tau]^{-1} -I}{\tau}\right)\right] \vvdot \left[ \nabla \,\,\left(\dfrac{\overline{\bu} _\tau -\bu}{\tau}\right)  \right]  \,  d\bx \nonumber \\
  &\qquad \qquad \qquad - \int_{\Omega\setminus\Gamma} \CC  \left[ \nabla \bu \,[D\phi_\tau]^{-1} \right] \vvdot \left[ \nabla \,\,\left(\dfrac{\overline{\bu} _\tau -\bu}{\tau}\right) \,\,\left(\dfrac{[D\phi_\tau]^{-1}-I)}{\tau}\right)\right]  
 \nonumber
\end{align}
Consequently, uniformly in $\tau\in [0,\tau_0]$,
\begin{equation} \label{eq:ABound}
|A|\lesssim C(L,\gamma)\, \|\bu\|_{H^1(\Omega\setminus \overline{S})}\; \|\,\,\dfrac{\nabla(\buBar_\tau-\bu)}{\tau}\|_{L^2(\Omega)}.
\end{equation}
For the second term, we recall the definition of $\buBar_\tau$ \eqref{eq:BuBarDef}  to obtain:
\begin{align}
  B &=  \int_{\Omega\setminus \Gamma} \frac{1}{\tau} \CC \left[\nabla \Breve{\bu}_\tau  [D\phi_\tau]^{-1}   \right]\vvdot \left[ \nabla \left(\,\,\dfrac{\overline{\bu} _\tau -\bu}{\tau}\right) [D\phi_\tau]^{-1}   \right] \, d\bx  \nonumber \\
  &\qquad - \int_{\Omega\setminus \Gamma} \CC \left[\nabla \,\,\dfrac{\ell_{\tau\,\bhtilde}}{\tau}  [D\phi_\tau]^{-1} \right]\vvdot \left[ \nabla \,\,\dfrac{(\overline{\bu} _\tau -\bu)}{\tau} [D\phi_\tau]^{-1}  \right] \, d\bx=: B_1+B_2. \nonumber
\end{align}
For the first term on the right, using the second equation in \eqref{eq:BuBarMinusBuTestFunction}:
\begin{align}
 B_1 &= \int_{\Omega\setminus \Gamma}   \CC_\tau \left[\nabla \Breve{\bu}_{\tau} \,\,\dfrac{([D\phi_\tau]^{-1}  -I)}{\tau}\right]\vvdot \left[ \nabla \,\,\dfrac{(\overline{\bu}_\tau -\bu)}{\tau} [D\phi_\tau]^{-1}   \right] \, d\bx \nonumber \\
 &\qquad - \int_{\Omega\setminus \Gamma}  \,\,\dfrac{(\CC_\tau-\CC)}{\tau} \left[\nabla \Breve{\bu}_{\tau}  [D\phi_\tau]^{-1} \right]\vvdot \left[ \nabla \,\,\dfrac{(\overline{\bu}_\tau -\bu)(\bx)}{\tau} [D \phi_\tau]^{-1}  \right] \, d\bx \nonumber \\
 &\qquad \qquad + \int_{\Omega\setminus \Gamma}  \CC_\tau \left[\nabla \Breve{\bu}_{\tau} [D\phi_\tau]^{-1} \right]\vvdot \left[ \nabla \,\,\dfrac{(\overline{\bu}_\tau -\bu)(\bx)}{\tau} [D \phi_\tau]^{-1}(\bx) \right] \, \frac{(J(\phi)-1)}{\tau} \, d\bx \nonumber
\end{align}
Hence, for $\tau\in [0,\tau_0]$ we have the bound:
\beq \label{eq:B1Bound}
 |B_1|\lesssim C(L,\gamma) \, \|\Breve\bu_\tau\|_{H^1(\Omega\setminus\overline{S})} \, \|\,\,\dfrac{\nabla(\buBar_\tau-\bu)}{\tau}\|_{L^2(\Omega)}
 \lesssim \tilde{C}(L,\gamma) \, \|\bu_\tau\|_{H^1(\Omega\setminus\overline{S_\tau})} \, \|\,\,\dfrac{\nabla(\buBar_\tau-\bu)}{\tau}\|_{L^2(\Omega)}.
\eeq
 For the second term on the right, we have similarly
 \beq \label{eq:B2Bound}
  |B_2|\lesssim C(L,\gamma) \, \|\ell_{\bhtilde}\\|_{H^1(\Omega\setminus\overline{S})} \, \|\,\,\dfrac{\nabla(\buBar_\tau-\bu)}{\tau}\|_{L^2(\Omega)},
 \eeq 
using also that by linearity $\ell_{\tau\,\bhtilde}/\tau=\ell_{\bhtilde}$.
From \eqref{eq:L2NormInitialBound}, \eqref{eq:ABound},\eqref{eq:B1Bound},\eqref{eq:B2Bound}, we then obtain the desired uniform bound:
\beq \label{eq:L2NormFinalBound}
  \|\,\,\dfrac{\nabla(\buBar_\tau-\bu)}{\tau}\|^2_{L^2(\Omega)} \lesssim
  C(L,\gamma,\|\bh\|_{H^\half_{00}(S)})\, \|\,\,\dfrac{\nabla(\buBar_\tau-\bu)}{\tau}\|_{L^2(\Omega)}, \qquad \forall \tau\in [0,\tau_0],
\eeq
where we have employed also the stability estimates for problems \eqref{eq:ForwardProblem}, \eqref{eq:ForwardProblemPerturbed}, and \eqref{eq:liftDef}.

We hence have that, upon passing to a subsequence, $\,\,\dfrac{\buBar_\tau-\bu}{\tau}$ converges weakly to an element
$\bv\in H^1_\Sigma(\Omega)$. By Rellich Compactness Theorem, the subsequence converges strongly in $L^2(\Omega)$ (in fact, in $H^s(\Omega)$ for $s<1$, and by weak continuity of the derivative:
\begin{equation*}
   \,\,\dfrac{\nabla(\buBar_\tau-\bu)}{\tau} \rightharpoonup \nabla \bv,
\end{equation*}
in $L^2(\Omega)$.
\end{proof}

\begin{remark}
For the validity of \eqref{eq:L2NormFinalBound}, it is not necessary to have $\CC$ globally Lipschitz in $\Omega$. It is sufficient that $\CC$ is uniformly Lipschitz in a neighborhood of $\Gamma$, which also contains $\Gamma_\tau$ for $\tau$ sufficiently small.
\end{remark}

We next establish uniqueness of the weak limit and strong convergence, by characterizing the material derivative as the solution of a certain variational problem.

\begin{theorem} \label{t:MatDerivative}
In the hypotheses of Lemma \ref{lem:DotuBarWeakConv},  
$\,\,\dfrac{\buBar_\tau-\bu}{\tau}$ converges as $\tau\to 0^+$ strongly to $\buBarDot\in H^1_\Sigma $, which is the unique solution of the variational problem:
\begin{align}
    \int_{\Omega} \mathbb{C} \nabla \dot{\overline{\bu}} \vvdot 
    \nabla \bw \, d\bx &= - \int_{\Omega\setminus\Gamma} \mathbb{B}
    \,\nabla\bu \vvdot \nabla \bw  \, d\bx  -\int_{\Omega \setminus \Gamma} \mathbb{C}  \,\nabla\bu \vvdot
    \nabla \bw \;\dive \cU  \, d\bx +\int_{\Omega \setminus \Gamma} \mathbb{C} \,\nabla\bu \vvdot \nabla \bw  \, D\cU \, d\bx \nonumber \\
    & +\int_{\Omega\setminus\Gamma}\mathbb{C} \,\nabla\bu  \,D \cU \vvdot 
    \nabla \bw  \, d\bx - \int_{\Omega\setminus \Gamma} \mathbb{C}\, \nabla \bell_{\bhtilde} \vvdot \nabla \bw  \, d\bx, \label{eq:buDotBarWeakForm}
\end{align}
for $\bw \in H^1_\Sigma(\Omega)$, where 
\begin{align}
    \mathbb{B}^{abcd}:=\lim_{\tau \to 0^+} \,\,\dfrac{\CC^{abcd}_\tau-\CC^{abcd}}{\tau}= \pa^e_{\bx} \CC^{abcd} \cU_e, 
    \label{eq:BBdef}
\end{align}
in $L^\infty(\Omega)$.
\end{theorem}

\begin{remark}
In general, the 4th-order tensor $\BB$ does not enjoy the same symmetry properties as $\CC$, which reflects the fact that the system of elastostatics is not covariant under changes of coordinates.
\end{remark}

\begin{proof}
We begin by showing that any weak limit of $\,\,\dfrac{\buBar_\tau-\bu}{\tau}$ satisfies \eqref{eq:buDotBarWeakForm}. This fact will also ensure its uniqueness.

We start by writing $a_\tau(\,\,\dfrac{\buBar_\tau-\bu}{\tau},\bw)$ for a generic test function $\bw\in H^1_\Sigma(\Omega)$.
We proceed as in the computation of terms $A$ and $B$ in the proof of Lemma \eqref{lem:DotuBarWeakConv},  again exploiting the weak formulation  \eqref{eq:WeakFormOriginalDomain} for $\Breve{\bu}_\tau$:
\begin{align}
  &- \int_{\Omega}  \CC_\tau \left[ \,\,\dfrac{(\nabla \buBar_\tau - \nabla\bu)}{\tau} 
 [D\phi_\tau]^{-1}\right] \vvdot \left[\nabla \bw [D \phi_\tau]^{-1}\right]
 \, J(\phi_\tau) \,d\bx =  \,\,\dfrac{1}{\tau}
 \int_{\Omega\setminus\Gamma} \mathbb{C}_\tau \left[ \nabla\bu
[D \phi_\tau] ^{-1} \right] \vvdot \left[\nabla \bw 
[D \phi_\tau ]^{-1}\right]   \nonumber \\
   & \qquad \qquad \qquad \, J(\phi_\tau) \,d\bx - \int_{\Omega\setminus \Gamma} \mathbb{C}_\tau
\left[ \,\,\dfrac{(\nabla\Breve{\bu}_{\tau} - \nabla \bell_{\tau \,  \bhtilde})}{\tau}   [D\phi_\tau]^{-1}\right] \vvdot  \left[\nabla \bw D [\phi_\tau ]^{-1} \right]\, J(\phi_\tau) \,d\bx= \nonumber \\
&  \,\,\dfrac{1}{\tau} \int_{\Omega\setminus\Gamma} \mathbb{C}_\tau \left[\nabla \bu  [D \phi_\tau]^{-1}\right] \vvdot \left[ \nabla \bw  [D \phi_\tau]^{-1}\right] \, J(\phi_\tau) \,d\bx 
+ \int_{\Omega\setminus \Gamma} \mathbb{C}_\tau \left[ \bell_{\bhtilde}    
[D \phi _\tau] ^{-1}\right]
 \vvdot  \left[\nabla \bw [D \phi_\tau]^{-1}  \right]  \, J(\phi_\tau) \,d\bx.
\end{align}
For notational convenience, we write the above identity abstractly as:
 \; $   G_\tau=E_\tau+F_\tau$.

Our goal is to pass to the limit $\tau\to 0_+$ in each term. In all three terms, we will use the assumptions on $\CC$ and $\phi_\tau$, in particular \eqref{eq:DphiTauDer} and \eqref{eq:JphiTauDer}. We also recall that the composition of two Lipschitz maps is still Lipschitz with constant that can be bounded by the product of the two Lipschitz constants, which implies in particular that 
$\CC_\tau \to \CC$ in $W^{1,\infty}(\Omega)$ as $\tau\to 0$. Then, we can compute the tensor $\BB$ as follows:
\begin{align}
      \dfrac{(\CC_\tau(\bx) - \CC(\bx))}{\tau} &=\, \dfrac{\mathbb{C}(\bx + 
      \tau \mathcal{U}(\bx)) - \mathbb{C}(\bx)}{\tau}  \nonumber \\
      &= \,\dfrac{\tau\,\mathcal{U}(\bx) \cdot \nabla_{\bm{z}} \mathbb{C}(\bm{z}(\tau)) }{\tau} 
      \underset{\tau\to 0_+}{\longrightarrow} \cU(\bx)\ccdot \nabla_{\bx} \mathbb{C}(\bx),
\end{align}
where $\bz(\tau)$ is a point on the segment joining  $\bx$ with $\bx + 
\tau \mathcal{U}(\bx)$.

We first investigate the behavior as $\tau\rightarrow  0_+$ of  the integral $G_\tau$ on the left-hand side. By Lemma \ref{lem:DotuBarWeakConv} the left-hand side consists of the pairing between an $L^2$ weakly convergent (sub)sequence $ \,\,\dfrac{\nabla \buBar_\tau-\nabla \bu}{\tau}$ and an $L^2$ strongly convergent sequence $J(\phi_\tau)\,\left(\CC_\tau \nabla \bw [D \phi _\tau] ^{-1}\right) [D \phi _\tau]^{-1}$. This last expression indicates the action of a 4th-order tensor onto a 2nd-order tensor  yielding a 2nd-order tensor (up to the scalar factor $J(\phi_\tau)$), and strong convergence follows from H\"older's inequality. Then, by weak-strong convergence:
\begin{align} 
  G_\tau&= - \int_{\Omega}  \CC_\tau \left[ \,\,\dfrac{(\nabla \buBar_\tau - \nabla\bu)}{\tau} 
 [D\phi_\tau]^{-1}\right] \vvdot \left[\nabla \bw [D \phi_\tau]^{-1}\right]
 \, J(\phi_\tau) \,d\bx  \nonumber \\
 &\qquad  =-\int_{\Omega} J(\phi_\tau)\, \left(\CC_\tau \nabla \bw 
 [D\phi_\tau]^{-1}\right) [D \phi _\tau] ^{-1} \vvdot  \,\,\dfrac{(\nabla \buBar_\tau-\nabla \bu)}{\tau}\, d\bx
 \nonumber \\
 & \qquad \qquad \underset{\tau\to 0_+}{\longrightarrow} \quad  -\int_{\Omega} \CC \nabla \bw  \vvdot \nabla \buBarDot\, d\bx= 
 -\int_{\Omega} \CC \nabla \buBarDot \vvdot \bw\, d\bx. \label{eq:LHSconv}
\end{align}
Next, we tackle the first integral $E_\tau$ on the right-hand side. We split into three parts, taking advantage of the fact that $\bu$ solves \eqref{eq:WeakForm}:
\begin{align}
  E_\tau & = \,\,\dfrac{1}{\tau}
 \int_{\Omega\setminus\Gamma} \mathbb{C}_\tau \left[ \nabla\bu
[D \phi_\tau] ^{-1} \right] \vvdot \left[\nabla \bw 
[D \phi_\tau ]^{-1}\right] \, J(\phi_\tau) \,d\bx \nonumber \\
 & = \int_{\Omega\setminus\Gamma} \,\,\dfrac{\left(\mathbb{C}_\tau -\CC\right)}{\tau}  \left[ \nabla\bu
[D \phi_\tau] ^{-1} \right] \vvdot \left[\nabla \bw 
[D \phi_\tau ]^{-1}\right] \, J(\phi_\tau) \,d\bx \nonumber \\
& \qquad + \left\{\int_{\Omega\setminus\Gamma} \CC \left[ \nabla\bu
[D \phi_\tau] ^{-1} \right] \vvdot \left[\nabla \bw 
 \,\dfrac{\left([D \phi_\tau ]^{-1} -I\right)}{\tau} \right] \, J(\phi_\tau) \,d\bx  \right. \nonumber \\
  &\qquad \qquad \left. + \int_{\Omega\setminus\Gamma} \CC \left[ \nabla\bu
 \,\,\dfrac{\left([D \phi_\tau ]^{-1} -I\right)}{\tau}
 \right] \vvdot \left[\nabla \bw [D \phi_\tau] ^{-1}
  \right] \, J(\phi_\tau) \,d\bx \right\}\nonumber  \\
  &\qquad \qquad \qquad + \int_{\Omega\setminus\Gamma} \CC \left[ \nabla\bu
 [D \phi_\tau] ^{-1} \right] \vvdot \left[\nabla \bw \right] \, \,\dfrac{\left(J(\phi_\tau)-1\right)}{\tau} \,d\bx \nonumber
\end{align}
Using once again the assumptions on $\CC$ and $\phi$, the regularity of $\bu$ and $\bw$, and H\"older's inequality, we can pass to the limit $\tau\to 0_+$ in each term above, as follows:
\begin{align}
 E_\tau & \underset{\tau\to 0_+}{\longrightarrow} \int_{\Omega\setminus\Gamma} \BB  \left[ \nabla\bu \right] \vvdot \left[\nabla \bw 
\right] \,d\bx \nonumber \\
&\qquad -  \left\{\int_{\Omega\setminus\Gamma} \CC \left[ \nabla\bu
D\cU \right] \vvdot \left[\nabla \bw \right] \, d\bx   +\int_{\Omega\setminus\Gamma} \CC \left[ \nabla\bu
 \right] \vvdot \left[\nabla \bw D\cU
  \right] \,  \,d\bx \right\}\nonumber \\
  & \qquad \qquad + \int_{\Omega\setminus\Gamma} \CC \left[ \nabla\bu
  \right] \vvdot \left[\nabla \bw \right] \, \dive \cU\,d\bx, 
  \label{eq:RHSconvB}
\end{align}
where we have also employed \eqref{eq:DphiTauDer}-\eqref{eq:JphiTauDer}, and defined $\BB$ by \eqref{eq:BBdef}.
Convergence in the last term $F_\tau$ follows straightforwardly from the convergence of $\phi_\tau$ to $I$ and of $\CC_\tau$ to $\CC$ in $W^{1,\infty}(\Omega)$:
\begin{align}
 F_\tau \underset{\tau\to 0_+}{\longrightarrow}
 \int_{\Omega\setminus \Gamma} \mathbb{C} \left[ \bell_{\bhtilde} \right]
 \vvdot  \left[\nabla \bw  \right]  \, d\bx. \label{eq:RHSconvC}
\end{align}
From \eqref{eq:LHSconv}-\eqref{eq:RHSconvC} it follows that any weak limit $\buBarDot$ must satisfy \eqref{eq:buDotBarWeakForm}.
Then the coercivity of the bilinear form $a_\Gamma$ in $H^1(\Omega)$ implies the uniqueness of $\buBarDot$ in this space.
Hence, the entire family $\big\{ \dfrac{\buBar_\tau - \bu}{\tau}\big\}_{0<\tau\leq \tau_0}$ converges weakly to $\buBarDot$ in $H^1(\Omega)$

We can now show that the convergence is actually strong in $H^1(\Omega)$.
We first observe that, since $\buBarDot\in H^1_\Sigma(\Omega)$, we can equivalently integrate on the right-hand side of \eqref{eq:buDotBarWeakForm} over $\Omega$ or $\Omega\setminus \Gamma$, and we can use $\buBarDot$ as test function $\bw$. Hence, we can write \eqref{eq:buDotBarWeakForm} in compact form as:
\begin{align}
  &a_\Gamma(\buBarDot,\buBarDot) = - \int_{\Omega\setminus\Gamma} \mathbb{B} \left(\nabla\bu \right) \vvdot \left(\nabla \buBarDot \right) \, d\bx 
  - a_\Gamma(\buBarDot,\buBarDot \,\dive \cU) + a_\Gamma(\buBarDot \,D\cU,\buBarDot) 
  + a_\Gamma(\buBarDot, \buBarDot \, D\cU) - a_\Gamma(\bell_{\bhtilde},\buBarDot), 
  \nonumber
\end{align}
where $a_\Gamma$ is defined below \eqref{eq:WeakForm}.
Next, similarly to \eqref{eq:L2NormInitialBound}, by Poincar\'e's inequality and the strong convexity of $\CC$, we have the estimate:
\begin{align}
 & \| \dfrac{\buBar_\tau - \bu}{\tau} - \buBarDot\|^2_{H^1(\Omega)} \lesssim \, \|\dfrac{\nabla \buBar_\tau - \nabla\bu}{\tau} - \nabla  \buBarDot \|_{L^2(\Omega)}^2  \nonumber \\
 &\lesssim  C(L,\bh) \, \int_{\Omega}\mathbb{C}_\tau\left[ \left(\,\dfrac{ \nabla (\overline{\bu}_\tau - \bu)}{\tau} - \nabla \buBarDot \right) D \phi_\tau ^{-1} \right] \vvdot \left[ \left(\,\dfrac{ \nabla (\overline{\bu}_\tau - \bu)}{\tau} - \nabla \buBarDot \right) D \phi_\tau ^{-1} \right] \, J(\phi_\tau)\, d\bx \nonumber \\
 & =a_\tau \left( \dfrac{\buBar_\tau - \bu}{\tau} - \buBarDot, \dfrac{\buBar_\tau - \bu}{\tau} - \buBarDot \right), \nonumber
\end{align}
with $a_\tau$ as in \eqref{eq:WeakFormOriginalDomain}.

We will show that the right-hand side of the expression above goes to zero as $\tau\to 0_+$. To this end, we split it into three parts:
\begin{equation} \label{eq:aTauBuDotSplit}
    a_{\tau}\left(\dfrac{\overline{\bu}_\tau -\bu}{\tau} -  \dot{\overline{\bu}},\dfrac{\overline{\bu}_\tau -\bu}{\tau} -  \dot{\overline{\bu}}\right) = a_{\tau}(\dot{\overline{\bu}},\dot{\overline{\bu}}) - 2a_{\tau}\left( \dfrac{\overline{\bu}_\tau -\bu}{\tau}, \dot{ \overline{\bu}} \right) + a_{\tau}\left(\dfrac{\overline{\bu}_\tau -\bu}{\tau},\dfrac{\overline{\bu}_\tau -\bu}{\tau}\right). 
\end{equation}
The convergence of $\CC_\tau$ to $\CC$ and of $\phi_\tau$ to $I$ in 
$W^{1,\infty}(\Omega)$ again implies that
\beq
  a_\tau(\bv_1,\bv_2) \underset{\tau \to 0_+}{\longrightarrow} 
  a_\Gamma(\bv_1,\bv_2), \quad \forall \bv_i \in H^1(\Omega\setminus\Gamma), \; i=1,2. \label{eq:aTautoaGammaConv}
\eeq
From \eqref{eq:aTautoaGammaConv}, we immediately have that
\beq
    a_{\tau}(\dot{\overline{\bu}},\dot{\overline{\bu}}) 
    \underset{\tau\to 0_+}{\longrightarrow}  a_{\Gamma} (\dot{\overline{\bu}},\dot{\overline{\bu}}),  \label{eq:FirstaTauTermConv}
\eeq
 Also, he first part of the proof using $\buBarDot$ as test function $\bw$ readily yields:
\beq
    a_{\tau}\left( \dfrac{\overline{\bu}_\tau -\bu}{\tau}, \dot{ \overline{\bu}} \right)
    \underset{\tau\to 0_+}{\longrightarrow} a_{\Gamma} (\dot{\overline{\bu}},\dot{\overline{\bu}}).  \label{eq:SecondaTauTermConv}
\eeq
We cannot directly pass to the limit $\tau\to 0_+$ in $a_{\tau}\left(\dfrac{\overline{\bu}_\tau -\bu}{\tau},\dfrac{\overline{\bu}_\tau -\bu}{\tau}\right)$, because we only have weak convergence of the arguments at this point. However, we recall the second identity in \eqref{eq:BuBarMinusBuTestFunction}, which allows us to rewrite this term as:
\begin{equation} \label{eq:ThirdaTauTermConv.0}
  a_{\tau}\left(\dfrac{\overline{\bu}_\tau -\bu}{\tau},\dfrac{\overline{\bu}_\tau -\bu}{\tau}\right) = 
 - a_\tau\left(\bell_{\bhtilde}, \dfrac{\overline{\bu}_\tau -\bu}{\tau} \right) 
- a_\tau\left(\dfrac{\bu}{\tau}, \dfrac{\overline{\bu}_\tau -\bu}{\tau}\right),
\end{equation}
using that $\dfrac{\bell_{\tau\bhtilde}}{\tau}= \dfrac{\tau\, \bell_{\bhtilde}}{\tau}=\bell_{\bhtilde}$.
Although $\bu$ and $\bell_{\bhtilde}$ are not test function, as they jump across $S$, the definition of $a_\tau$ only requires the entries to be in $H^1(\Omega\setminus\Gamma)$.
Consequently, we can pass to the limit $\tau\to 0_+$ in the first term on the right by \eqref{eq:aTautoaGammaConv} to obtain:
\beq  \label{eq:ThirdaTauTermConv.1}
  a_\tau\left(\bell_{\bhtilde}, \dfrac{\overline{\bu}_\tau -\bu}{\tau} \right)
  \underset{\tau\to 0_+}{\longrightarrow} a_\Gamma(\bell_{\bhtilde},\buBarDot).
\eeq
For the second term on the right, we employ arguments similar to those used in the first half of the proof, distributing the factor $1/\tau$ onto strongly converging terms:
\begin{align}
    a_\tau\left(\dfrac{\bu}{\tau}, \dfrac{\overline{\bu}_\tau -\bu}{\tau}\right) &= \int_{\Omega\setminus\Gamma} \CC_\tau \, \left[ \nabla \bu  \dfrac{[D \phi_\tau]^{-1}}{\tau} \right] \vvdot \left[ \nabla 
    \left(\dfrac{\overline{\bu}_\tau -\bu}{\tau}\right)  [D \phi_\tau]^{-1} \right] \, J(\phi_\tau)\, d\bx \nonumber \\
    &= \int_{\Omega\setminus\Gamma} \dfrac{\left(\CC_\tau-\CC\right)}{\tau} \, \left[ \nabla \bu  [D \phi_\tau]^{-1} \right] \vvdot \left[ \nabla 
    \left(\dfrac{\overline{\bu}_\tau -\bu}{\tau}\right)  [D \phi_\tau]^{-1} \right] \, J(\phi_\tau)\, d\bx \nonumber \\
    &\qquad +  \int_{\Omega\setminus\Gamma} \CC \, \left[ \nabla \bu  
    \left(\dfrac{[D \phi_\tau]^{-1}-I}{\tau}\right) \right] \vvdot \left[ \nabla \left(\dfrac{\overline{\bu}_\tau -\bu}{\tau}\right)  [D \phi_\tau]^{-1} \right] \, J(\phi_\tau)\, d\bx \nonumber \\
    &\qquad \qquad \qquad  +  \int_{\Omega\setminus\Gamma} \CC \, \left[ \nabla \bu \right] \vvdot \left[ \nabla \left(\dfrac{\overline{\bu}_\tau -\bu}{\tau}\right) \left(\dfrac{[D \phi_\tau]^{-1}-I}{\tau}\right) \right]  \, J(\phi_\tau)\, d\bx \nonumber \\
    &\qquad \qquad \qquad \qquad +  \int_{\Omega\setminus\Gamma} \CC \, \left[ \nabla \bu \right] \vvdot \left[ \nabla \left(\dfrac{\overline{\bu}_\tau -\bu}{\tau}\right)   \right] \, \left(\dfrac{J(\phi_\tau)-1}{\tau}\right)\, d\bx, \nonumber 
\end{align}
where in the last term we have used that $\bu$ solves \eqref{eq:WeakForm}.
Then, using again the assumptions on $\CC$ and $\phi_\tau$, \eqref{eq:DphiTauDer} and \eqref{eq:JphiTauDer}, we can pass to the limit $\tau\to 0_+$ in each term on the right, since we are pairing the weakly convergent sequence $\dfrac{\overline{\bu}_\tau -\bu}{\tau}$ with a strongly convergent sequence, obtaining that
\begin{align}
    a_\tau\left(\dfrac{\bu}{\tau}, \dfrac{\overline{\bu}_\tau -\bu}{\tau}\right) &\underset{\tau \to 0_+}{\longrightarrow}
    \int_{\Omega\setminus\Gamma} \mathbb{B}
    \,\nabla\bu \vvdot \nabla \buBarDot  \, d\bx  +\int_{\Omega \setminus \Gamma} \mathbb{C}  \,\nabla\bu \vvdot
    \nabla \buBarDot \;\dive \cU  \, d\bx -\int_{\Omega \setminus \Gamma} \mathbb{C} \,\nabla\bu \vvdot \nabla \buBarDot  \, D\cU \, d\bx \nonumber \\
    & \qquad \qquad \qquad -\int_{\Omega\setminus\Gamma}\mathbb{C} \,\nabla\bu  \,D \cU \vvdot 
    \nabla \buBarDot \, d\bx = - a_\Gamma(\buBarDot,\buBarDot)
    -a_\Gamma(\bell_{\bhtilde},\buBarDot), \label{eq:ThirdaTauTermConv.2}
\end{align}
given that $\buBarDot$ solves \eqref{eq:buDotBarWeakForm} for any test function $\bw\in H^1_\Sigma(\Omega)$.
From \eqref{eq:ThirdaTauTermConv.0}, using \eqref{eq:ThirdaTauTermConv.1} and \eqref{eq:ThirdaTauTermConv.2}, we have that
\begin{equation} \label{eq:ThirdaTauTermConv.3}
   a_{\tau}\left(\dfrac{\overline{\bu}_\tau -\bu}{\tau},\dfrac{\overline{\bu}_\tau -\bu}{\tau}\right) \underset{\tau \to 0_+}
   {\longrightarrow}
   - a_\Gamma(\bell_{\bhtilde},\buBarDot) + a_\Gamma(\buBarDot,\buBarDot)
    + a_\Gamma(\bell_{\bhtilde},\buBarDot) = a_\Gamma(\buBarDot,\buBarDot).
\end{equation}
Finally, by employing \eqref{eq:FirstaTauTermConv}, \eqref{eq:SecondaTauTermConv}, and \eqref{eq:ThirdaTauTermConv.3} in \eqref{eq:aTauBuDotSplit}, we can conclude that:
\begin{equation*}
   a_{\tau}\left(\dfrac{\overline{\bu}_\tau -\bu}{\tau} -  \dot{\overline{\bu}},\dfrac{\overline{\bu}_\tau -\bu}{\tau} -  \dot{\overline{\bu}}\right)
    \underset{\tau\to 0_+}{\longrightarrow}
    a_\Gamma(\buBarDot,\buBarDot) - 2 \, a_\Gamma(\buBarDot,\buBarDot) +  a_\Gamma(\buBarDot,\buBarDot)=0,
\end{equation*}
which in turn implies that
\begin{equation*}
   \dfrac{\overline{\bu}_\tau -\bu}{\tau}
   \underset{\tau\to 0_+}{\longrightarrow} \buBarDot \quad 
   \text{in } H^1_\Sigma(\Omega).
\end{equation*}
The proof is now complete.
\end{proof}

We recall that  
$\buBarDot_\tau=\Breve{\bu}_\tau-\bell_{\tau \bhtilde}$ and 
that $\dfrac{\bell_{\tau \bhtilde}}{\tau}=\bell_{\bhtilde}$ is constant in $\tau$. 
Then, from Theorem \ref{t:MatDerivative} we  straightforwardly obtain an expression for the material derivative 
\beq \label{eq:buDotDef}
   \Dot{\bu} := \lim_{\tau\to 0_+} \dfrac{\Breve{\bu}_\tau - \bu}{\tau},
\eeq
as solution of a variational problem.

\begin{corollary} \label{C:uMaterialDer}
Under the hypotheses of Lemma \ref{lem:DotuBarWeakConv},    
$\,\,\dfrac{\Breve{\bu}_\tau -\bu}{\tau}$ converges to  
    $\dot{\bu} =  \dot{\overline{\bu}} + \bell_{\bhtilde}$ strongly in $H^1_ \Sigma(\Omega \setminus\overline{S})$ and $\dot{\bu}$ is the unique variational solution of the problem
\begin{align} 
    &\int_{\Omega \setminus \overline{S}} \mathbb{C} \nabla \dot{\bu} \vvdot  \nabla \bv \,d\bx
    = - \int_{\Omega\setminus\overline{S}} \mathbb{B} \nabla\bu \vvdot \nabla \bv  \,d\bx \nonumber \\
    &-\int_{\Omega \setminus \overline{S}} \mathbb{C}  \nabla\bu \vvdot  \nabla \bv \, \dive \mathcal{U} \,d\bx + \int_{\Omega \setminus \overline{S}} \mathbb{C} \nabla\bu \vvdot \nabla \bv \, D \mathcal{U} \,d\bx + \int_{\Omega\setminus\overline{S}} \mathbb{C} \nabla\bu \, D \mathcal{U} \vvdot  \nabla \bv \,d\bx, \label{eq:buDotWeakForm}
\end{align}
for all $\bv\in H^1_\Sigma(\Omega)$, and  $[\dot{\bu}]_S = \bh$.
\end{corollary}

Above we have used the stability estimates for the forward problem, which imply the trace in $H^{1/2}(\Gamma)$ is continuous with respect to $\tau$. We have also used that
\begin{equation*}
    [\dfrac{\Breve{\bu}_\tau - \bu}{\tau}]_{\Gamma}=\bhtilde,
\end{equation*}
with $\bhtilde= \bzero$ on $\Gamma \setminus \overline{S}$, and employed Lemma \ref{lem: Htilde}.

Now that we have established the existence of the material derivative and derived a formula for it, we will use it to compute the shape derivative, defined in \eqref{eq:ShapeDerDef}.

\begin{theorem} \label{t:ShapeDerivative}
Under the hypotheses of Lemma \ref{lem:DotuBarWeakConv}, if $\bu$ is the unique weak solution of problem \eqref{eq:ForwardProblem} and 
$\bgtilde_\tau$, $S_\tau$ are defined as in \eqref{eq:SlipDeformationDef}, \eqref{eq:GammaTauDef}, respectively, the shape derivative of the functional $\cJ$ is given by:
\begin{align} 
   \dfrac{d}{d\tau} \cJ\left(S_\tau,\bg_\tau\right)\lfloor_{\tau=0} &= 
   - \int_{\Omega\setminus\overline{S}} \CC \left[ \nabla \bu \, D\cU\right] \vvdot \nabla \bw\, d \bx 
   - \int_{\Omega\setminus\overline{S}} \CC \nabla \bu \vvdot \left[ \nabla \bw \, D\cU\right]\, d \bx \nonumber \\
   & + \int_{\Omega\setminus\overline{S}} \CC \nabla \bu \vvdot \nabla \bw\, \dive \cU \, d \bx 
   + \int_S \left[\CC(\nabla \bw) \,\bh\right] \cdot \bn\, d\sigma(\bx) 
    - \int_{\Omega\setminus\overline{S}}  \BB \nabla \bu \vvdot \nabla \bw \, d\bx, \label{eq:ShapeDerFormula}
\end{align}
where $\bw\in H^1_\Sigma(\Omega)$ is the unique weak solution of the adjoint problem:
\begin{align} \label{eq:AdjointProblem}
    \begin{cases}
        \dive \left(\CC \nabla \bw\right) = \bzero, & \text{in } \Omega, \\
        (\CC\nabla \bw)\bnu = (\bu-\bu_m)\,\chi_{\Xi}, & \text{on } \pa\Omega\setminus \Sigma,\\
        \bw = \bzero, & \text{on } \Sigma.
    \end{cases}
\end{align}
\end{theorem}

\begin{remark}
 Using elementary properties of the trace of a matrix, we can rewrite the second integral on the right in \eqref{eq:ShapeDerFormula} as 
\begin{equation*}
  - \int_{\Omega\setminus\overline{S}} \CC \left[ \nabla \bu \, D\cU^T\right] \vvdot \nabla \bw\, d \bx 
\end{equation*}
 where the superscript $T$ denotes the transpose of a matrix. Hence, \eqref{eq:ShapeDerFormula} takes the compact form:
 \beq  \label{eq:PolarizationTensor}
    \dfrac{d}{d\tau} \cJ\left(S_\tau,\bg_\tau\right)\lfloor_{\tau=0} = 
    \int_{\Omega\setminus \overline{S}} \mathbb{M} \left[ \nabla \bu \right] \vvdot \nabla \bw \, d\bx +
     \int_S \left[\CC(\nabla \bw) \,\bh\right] \cdot \bn\, d\sigma(\bx), 
 \eeq
 where the so-called {\em polarization tensor} $\mathbb{M}$ encodes the distributed effect of the infinitesimal movement of the fault.
 $\MM$ is given in terms of its action on matrices by
 \beq \label{eq:PolarizationTensorDef}
     \mathbb{M}[\mathbf{A}] := \BB[\mathbf{A}] -  \CC \left[  \mathbf{A} (D\cU^T +D\cU) \right] + (\dive \cU) \CC[\mathbf{A}],
 \eeq
 for $\mathbf{A}$ any $d\times d$ matrix. 
 Note that, in general, $\mathbb{M}$ does not enjoy the same symmetries as $\CC$ does, unless $\CC$ is constant.
\end{remark}

\begin{proof}
Since $\dfrac{\bu_\tau-\bu}{\tau}$ converges strongly to $\buDot$ in $H^1_\Sigma(\Omega)$ by Corollary \ref{C:uMaterialDer}, hence in $L^2(\pa\Omega)$ by the Trace Theorem, we can differentiate under the integral sign in the definition of the functional $\cJ$:
\begin{equation*}
  \dfrac{d}{d\tau} \cJ\left(S_\tau,\bg_\tau\right)\lfloor_{\tau=0} = \int_\Xi ( \bu - \bu_m)\cdot \buDot \, d\sigma(\bx),
\end{equation*}
where we used the definition of the material derivative \eqref{eq:buDotDef}.

Let now $\bw$ solve the adjoint problem \eqref{eq:AdjointProblem} (this is a standard variational problem that admits a unique weak solution by the Lax-Milgram Theorem, for instance). Then, we can equivalently write the expression above on the right as:
\begin{equation*}
    \dfrac{d}{d\tau} \cJ\left(S_\tau,\bg_\tau\right)\lfloor_{\tau=0} = \int_{\pa\Omega} \left[\CC(\nabla \bw) \, \bnu\right] \cdot \buDot\,d \sigma(\bx).
\end{equation*}
We recall that we can integrate over $\Omega\setminus \Gamma$ or $\Omega\setminus \overline{S}$, because neither $\bu$ nor $\bw$ jump across $\Gamma\setminus \overline{S}$, and that $\Gamma$ splits $\Omega$ into an outer region $\Omega^+$ that touches $\pa\Omega$ and an inner core $\Omega^-$. From \eqref{eq:AdjointProblem}, although $\buDot$ cannot be used as a test function, integrating first in $\Omega^+$ and then in $\Omega^-$ gives that
\begin{equation*}
\int_{\pa\Omega} \left[\CC(\nabla \bw) \, \bnu\right] \cdot \buDot\,d \sigma(\bx)= - \int_{\Omega^+} \CC \nabla \bw \vvdot \nabla \buDot \, d\bx -  \int_{\Omega^-} \CC \nabla \bw \vvdot \nabla \buDot \, d\bx + \int_{\Gamma} \left[\CC(\nabla \bw) \, \bn\right] \cdot [\buDot]_{\Gamma}\,d \sigma(\bx),
\end{equation*}
recalling that by convention $\bn$ is the outer normal to $\Omega_-$. But we know that the jump $[\buDot]_{\Gamma} = \bhtilde$ (so, in fact, $\buDot$ only jumps across $S$), so that we finally have:
\begin{equation*}
   \dfrac{d}{d\tau} \cJ\left(S_\tau,\bg_\tau\right)\lfloor_{\tau=0} = -\int_{\Omega\setminus \overline{S}} \CC \nabla \bw \vvdot \nabla \buDot \, d\bx + \int_{S} \left[\CC(\nabla \bw) \, \bn\right] \cdot \bh\,d \sigma(\bx),
\end{equation*}
where we also used that $\bw$ can be taken as test function for \eqref{eq:buDotWeakForm}.
\end{proof}

We close this section by deriving a {\em boundary shape derivative} from the distributed shape derivative in case $\CC$ is constant, i.e., the rock is homogeneous in the region of interest around the fault. While this formula, which contains only integrals over the fault $S$, may lead to a less stable numerical implementation, it allows a more transparent geometric interpretation. We work with strong solutions of \eqref{eq:ForwardProblem} and \eqref{eq:AdjointProblem}, assuming that $\pa\Omega$, $\S$ and $\bg$, $\bh$ are regular enough. By a strong solution $\bu$ of \eqref{eq:ForwardProblem}, we mean that $\bu \in H^2(\Omega\setminus \overline{S})\cap 
H^1_\Sigma(\Omega\setminus\overline{S})$ and that the problem  \eqref{eq:ForwardProblem} is attained at least pointwise a.e. Similarly, a strong solution $\bw\in H^2(\Omega)\cap H^1_\Sigma(\Omega)$ solves \eqref{eq:AdjointProblem} at least pointwise a.e.\,. From standard regularity theory, weak solutions are strong solutions if $\pa\Omega$ and $S$ are at least of class $C^{1,\alpha}$, $\alpha>0$, and $\bg, \bh \in H^{s}_0(S)$, $s\geq 3/2$.
We only compute the derivative in 2 space dimensions to simplify the resulting expression.

Our starting point is formula \eqref{eq:PolarizationTensor}. We show that we can write $\MM  [\nabla \bu] \vvdot \nabla \bw$ as the divergence of a certain vector $\bb$. Since we assume that $\CC$ is constant, $\BB\equiv 0$. Furthermore, since both $\bu$ and $\bw$ satisfy the homogeneous system of elasticity a.e. in $\Omega\setminus \Gamma$, it it enough to show that:
\begin{align}
 \MM  [\nabla \bu] \vvdot \nabla \bw &= -\dive \bb + \left((\cU \cdot\nabla) \bw\right) \cdot \left(\dive [\CC \nabla \bu] \right)+ \left((\cU\cdot\nabla) \bu\right) \cdot \left(\dive [\CC \nabla \bw]\right), \nonumber
\end{align}
for a suitable vector field $\bb$.
Next, we exploit some elementary tensor calculus. If $\bF$ is a matrix and $\ba$ is a vector, by $\ba\,\bF$ we mean vector-matrix multiplication (suppressing the transpose for notational ease), that is: \ $[\ba \,\bF]_j= \ba_i\, \bF^i_j$, using Einstein's summation convention. Then, it is easy to see that
\beq \label{eq:DivOfProduct}
    \dive(\ba \bF) = \ba \cdot \dive \bF +\nabla \ba \vvdot \bF,
\eeq
where \ $ [\dive \bF]_i=\pa_j\bF^{ij}= \text{Tr\,} (\nabla \bF)^i$.
We let, following the convention above:
\begin{align} 
     &\bb = \ba_1 \bF_1 + \ba_2 \bF_2 - a\,\cU, \qquad \text{with} \nonumber \\
     &\ba_1 = (\cU\cdot \nabla) \bw, \quad \ba_2 = (\cU\cdot \nabla) \bu, \quad  a = \CC[\nabla \bu]\vvdot \nabla \bw,\nonumber \\
     &\bF_1  = \CC[\nabla \bu], \quad \bF_2  = \CC[\nabla \bw],
  \label{eq:VectorbDef}
\end{align}
where $a$ is a scalar field and we have identified the map $\cU$ with a vector field.
We apply \eqref{eq:DivOfProduct} to $\ba_1 \bF_1$:
\begin{equation*}
   \dive(\ba_1\, \bF_1) = \left((\cU\cdot \nabla) \bw\right) \left(\dive [\CC \nabla \bu]\right) + \CC [ \nabla \bu] \vvdot \nabla \left((\cU\cdot \nabla)\bw\right),
\end{equation*}
and note that 
\begin{align}
    [\nabla \left((\cU\cdot \nabla)\bw\right) ]^i_j &= \pa_j (\cU^\ell \pa_\ell \bw^i)  \quad \Rightarrow \quad \nonumber \\
    \nabla \left((\cU\cdot \nabla)\bw\right) &= [\nabla \bw] D\cU +(\cU\cdot \nabla) \nabla \bw. \nonumber
\end{align}
Consequently,
\beq \label{eq:a1F1Div}
 \dive(\ba_1\, \bF_1) = (\left(\cU\cdot \nabla) \bw\right) \left(\dive [\CC \nabla \bu]\right) + \CC [ \nabla \bu] \vvdot [\nabla \bw \,D\cU] + \CC[\nabla \bu] \vvdot \left[ (\cU\cdot \nabla) \nabla \bw\right].
\eeq 
Similarly, 
\begin{align} 
 \dive(\ba_2\, \bF_2) &= (\left(\cU\cdot \nabla) \bu\right) \left(\dive [\CC \nabla \bw]\right) + \CC [ \nabla \bw] \vvdot [\nabla \bu \,D\cU] + \CC[\nabla \bw] \vvdot \left[ (\cU\cdot \nabla) \nabla \bu\right] \nonumber \\
 &=   \left((\cU\cdot \nabla) \bu\right) \left(\dive [\CC \nabla \bw]\right) + \CC [ \nabla \bu D\cU] \vvdot \nabla \bw + \CC[\nabla \bw] \vvdot \left[ (\cU\cdot \nabla) \nabla \bu\right],
 \label{eq:a2F2Div}
\end{align}
where in the last identity we have used the symmetry of $\CC$.
We also recall that \  $\dive(a \,\cU) = a\,\dive \cU +\cU\cdot \nabla a$. Using this formula and \eqref{eq:a1F1Div}-\eqref{eq:a2F2Div}, 
we can write
\begin{align} 
 \dive (\bb) &= (\left(\cU\cdot \nabla) \bw\right) \left(\dive [\CC \nabla \bu]\right) + (\left((\cU\cdot \nabla) \bu\right) \left(\dive [\CC \nabla \bw]\right) + \CC [ \nabla \bu] \vvdot [\nabla \bw \,D\cU]  \nonumber \\
 & \qquad \qquad + \CC \left[ (\cU\cdot \nabla) \nabla \bu\right] \vvdot [\nabla \bw] + \CC[\nabla \bu] \vvdot \left[ (\cU\cdot \nabla) \nabla \bw\right] + \CC[\nabla \bw] \vvdot \left[ (\cU\cdot \nabla) \nabla \bu\right] \nonumber \\
 & \qquad \qquad\qquad -(\dive \cU) \CC[\nabla \bu]\vvdot \nabla \bw - (\cU\cdot \nabla)
 \left( \CC[\nabla \bu] \vvdot \nabla \bw\right), \label{eq:VectorbDivExpanded}
\end{align}
pointwise a.e in $\Omega\setminus \Gamma$.

Now, since $\CC$ is constant, exploiting again the symmetry of $\CC$,
\begin{equation*}
  \cU^i \pa_i \left( \CC[\nabla \bu] \vvdot \nabla \bw\right)= \cU^i \, \CC[\nabla \bu]\vvdot \pa_i \nabla \bw + \cU^i\, 
  \CC[\pa_i \nabla \bu]\vvdot \nabla \bw =  \CC[\nabla \bu]\vvdot [(\cU^i \pa_i) \nabla \bw] + \CC[\nabla \bw] \vvdot [
  (\cU^i \pa_i) \nabla \bu],
\end{equation*}
so the last term in \eqref{eq:VectorbDivExpanded} cancels other two terms in that expression and we conclude that
\begin{align}
    \dive (\bb) &= (\left(\cU\cdot \nabla) \bw\right) \left(\dive [\CC \nabla \bu]\right) + (\left((\cU\cdot \nabla) \bu\right) \left(\dive [\CC \nabla \bw]\right)   \nonumber \\
 & \qquad \qquad + \CC [ \nabla \bu] \vvdot [\nabla \bw \,D\cU]+ \CC \left[ (\cU\cdot \nabla) \nabla \bu\right] \vvdot [\nabla \bw] 
  -(\dive \cU) \CC[\nabla \bu]\vvdot \nabla \bw. \label{eq:VectorbDivFinal}
\end{align}
We can now derive the boundary shape derivative. We introduce some convenient notation. Let $\{\bt,\bn\}$ be  a curvilinear frame on $\Gamma$, hence on $S$, with $\bn$ coinciding with the unit outer normal to $\Omega^-$.
We denote the tangential projection by:
\begin{equation*}
   \bF_{\tang} :=\bF \bt, \qquad \ba_{\tang} := \ba \cdot \bt,
\end{equation*}
where $\bF$ is a matrix and $\ba$ is a vector, and according to our notation $\bF \bt$ denotes matrix-vector multiplication. 
Similarly, we denote the normal projection by:
\begin{equation*}
  \bF_{n} :=\bF \bn, \qquad \ba_n := \ba \cdot \bn.
\end{equation*}
We also employ the short-hand notation:
\begin{equation*}
   \pa_{\tang} := \bt\cdot \nabla, \quad \pa_n := \bn \cdot \nabla. 
\end{equation*}

\begin{theorem} \label{t:BoundaryShapeDer}
    In the hypotheses of Lemma \ref{lem:DotuBarWeakConv}, if in addition $\Omega\subset \RR^2$  and the fault $S$ are of class $C^{1,\alpha}$, $\alpha>0$, and $\CC$ is constant, the boundary shape derivative of the functional $\cJ$ is given by:
    \begin{equation} \label{eq:BoundaryShapeDerFormula}
        \dfrac{d}{d\tau} \cJ(S_\tau,\bg_\tau)\lfloor_{\tau=0}  = 
        \int_S \left(\CC[\nabla \bw]\right)_n (\cU_{\tang} \pa_{\tang} \bg) -  \cU_n \,\left(\CC[\nabla \bw]\right)_{\tang} (\pa_{\tang} \bg)+ \left(\CC[\nabla \bw]\right)_n (\bh) \, d\sigma(\bx),
    \end{equation}
    where $\bu$ is a strong solution of \eqref{eq:ForwardProblem} and $\bw$ is a strong solution of \eqref{eq:AdjointProblem}.
\end{theorem}

\begin{proof}
Since $\bu$ solves \eqref{eq:ForwardProblem} and $\bw$ solves \eqref{eq:AdjointProblem}, from \eqref{eq:VectorbDivFinal} it follows by integrating by parts in $\Omega^{\pm}$ that 
\begin{align}
       \dfrac{d}{d\tau} \cJ(S_\tau,\bg_\tau)\lfloor_{\tau=0}  &= -\int_{\Omega\setminus \Gamma} \dive \bb\, d\bx + \int_S \left(\CC (\nabla \bw)\right)_n \cdot \bh \,d\sigma(\bx)  \nonumber \\
       &= \int_{\Gamma} \left[(\bb)_n \right]_\Gamma \, d\sigma(\bx) + \int_S \left(\CC (\nabla \bw)\right)_n \cdot \bh \,d\sigma(\bx)\nonumber \\
        &= \int_\Gamma \left\{ \left[\left(\CC (\nabla \bw)\right)_n\right]_{\Gamma} \cdot (\cU\cdot \nabla) \bw + \left(\CC (\nabla \bw)\right)_n \cdot \left[(\cU\cdot \nabla) \bu \right]_\Gamma \right. \nonumber \\
        &\qquad \qquad \qquad \left.- \cU_n \, \left[ \CC(\nabla \bu) \vvdot \nabla \bw\right]_\Gamma
        \right\}\, d\sigma(\bx) + \int_S \left(\CC (\nabla \bw)\right)_n \cdot \bh \,d\bx. \label{eq:BoundaryShapeDerFormulaReduced}
\end{align}
Above we have used that $\cU$, $\nabla \bw$, and $\CC(\nabla \bw)_n$ do not jump across $\Gamma$. Indeed, $\cU$ is Lipschitz continuous on the whole $\Omega$, while $\bw \in H^2(\Omega)$, so $[(\nabla \bw)_\tang]_{\Gamma}= \bzero$ since $\bw$ is continuous and $\Gamma$ is regular. Also, the jump in the traction $\CC(\nabla \bw)_n$ across $\Gamma$ is zero from \eqref{eq:AdjointProblem} integrating by parts in $\Omega^{\pm}$. Then, as in \cite{MR04}, one can conclude from these two facts that $[(\nabla \bw)_n]_{\Gamma}= \bzero$, hence the whole gradient of $\bw$ does not jump across $\Gamma$.
On the other hand, similarly, we have
\begin{flalign}
     \left[(\cU\cdot \nabla) \bu \right]_\Gamma &= \cU_\tang \, \pa_\tang \left[ \bu \right]_\Gamma= \pa_\tang \bgtilde, \nonumber \\
      \left[ \CC(\nabla \bu) \vvdot \nabla \bw\right]_\Gamma & = \left[ \CC(\nabla \bu) \right]_\Gamma \vvdot \nabla \bw=
     \left(\left[ \CC(\nabla \bu)  \right]_\Gamma \right)_\tang \vvdot (\pa_\tang \bw)  \nonumber \\
      \,  &= \left(\CC(\nabla \bw)\right)_\tang   \vvdot  \pa_\tang [\bu]_{\Gamma} = \left(\CC(\nabla \bw)\right)_\tang   \vvdot  \pa_\tang \bgtilde,\nonumber
\end{flalign}
where in the last identity we used that the traction $\CC(\nabla \bu)_n$ does not jump across $\Gamma$ from \eqref{eq:ForwardProblem} and the symmetry of $\CC$. Inserting these expression into \eqref{eq:BoundaryShapeDerFormulaReduced} finally gives formula \eqref{eq:BoundaryShapeDerFormula}.
\end{proof}

\section{Numerical implementation} \label{s:numerics}

In this section, we implement a reconstruction algorithm based on a gradient descent method that uses the shape derivative \eqref{eq:ShapeDerFormula} of the functional defined in \eqref{eq:JfunctDef}.
While the algorithm is applicable in both two and three space dimensions, our numerical experiments are done in the two-dimensional case, and hence we focus on describing the method in 2D. The two-dimensional problem can be seen as looking at vertical slices of the full 3D problem and neglecting any stress across different slices. This is a simplification, but our numerical experiments are a proof of concept that the algorithm can recover the fault and the slip under suitable assumptions. We reserve to address a more realistic set-up for the numerical implementation in future work.

The reconstruction procedure is based on the gradient descent algorithm proposed in \cite{BMPS18} to reconstruct a piecewise-constant conductivity that jumps across a polygonal partition  in a composite material from current-to-voltage boundary measurements. There are significant differences between the case considered in \cite{BMPS18} and here. First of all, we deal with a system of PDEs and not a scalar PDE. Furthermore, we have non-homogeneous jumps across the fault. In particular the solution is not $H^1$ regular across the fault. Lastly, we measure only on part of boundary and we measure only once (twice in some unstable cases). These additional challenges make the numerical implementation more sensitive, especially to how the jump in the displacement $\bu$ across the fault is treated.

We concentrate primarily on recovering the fault geometry $S$, assuming the slip field $\bg$ is known. Indeed, once $S$ is determined, the problem of reconstructing $\bg$ is linear. We will consider both situations addressed in our prior works \cite{ABMdH19, ABM23}, that is, the case when  $\bg$ vanishes at the endpoints of the line $S$ and is in $H^{\frac{1}{2}}_{00}(S)$, and the case that $\bg$ is constant on $S$, even though the theoretical framework for the distributed shape derivative exploits a variational formulation and hence applies to the first case.

To solve both the forward and adjoint problems (see \eqref{eq:ForwardProblem} and \eqref{eq:AdjointProblem}), we employ the Discontinuous Galerkin (DG) method, which is non conforming and lends itself to tackling problems with inhomegenous jumps. Other numerical methods have been used for shape optimization. In the context of elasticity, among the many results we mention the recent works  \cite{RLL20, KP22}, where the boundary immerse method and immersed interface finite elements are employed.

\begin{figure}[H] 
\includegraphics[scale=0.35]{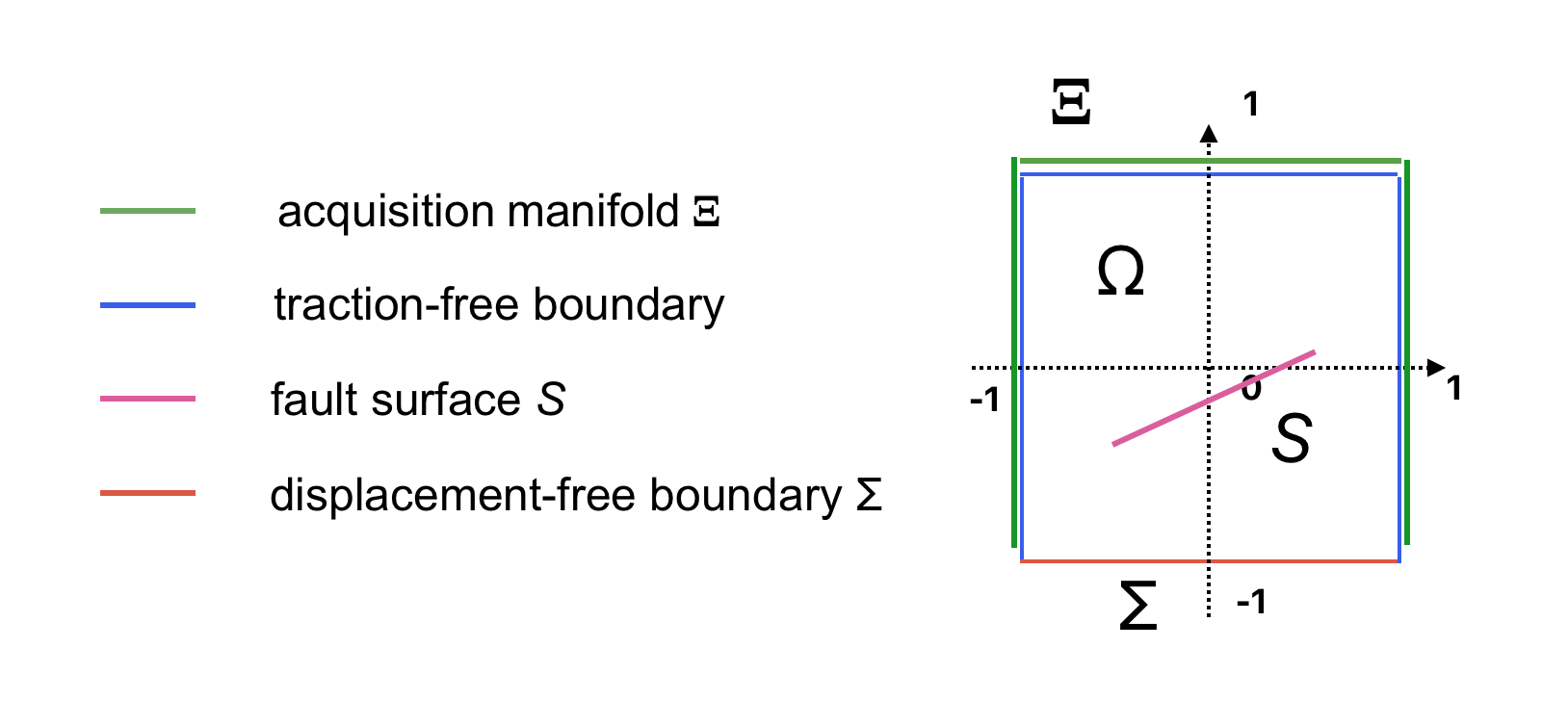}
\caption{Geometric set up for the numerical experiments.} \label{f:NumericalSetUp}
\end{figure}

We next discuss the DG method in the context of our problem.
For the numerical implementation, we let the domain $\Omega = [-1,1] \times [-1,1]$ and $S$ a segment, positioned away from the boundary. Because of the inherent instabilities in the problem, although unrealistic for the geophysical problem  that motivates our work, we find that we need to measure on a sufficiently large part of the boundary. Recall that the acquisition manifold $\Xi$ where the measurements are performed lies on the traction-free part of the boundary. Hence, we impose homogeneous Neumann conditions on three sides of the square $\Omega$ and homogeneous Dirichlet conditions on the fourth side identified with the set $\Sigma=\{(x,y)\ ; y=1, -1\leq x\leq 1\}$ (See Figure \ref{f:NumericalSetUp}).

\subsection{Discontinuous Galerkin (DG) method}

We consider a family of shape-regular partitions $\mathcal{T}_h$ where $0<h\ll 1$. These partitions consist of non-overlapping triangles $\mathcal{K}$ such that $\overline{\Omega}=\cup_{\mathcal{K}\in \mathcal{T}_h}\overline{\mathcal{K}}$. Specifically, $h=\max_{\mathcal{K}\in \mathcal{T}_h}h_{\mathcal{K}}$, where $h_{\mathcal{K}}=\textrm{diam}(\mathcal{K})$. 
The mesh conforms to the fault $S$ as follows.  We construct a trapezoidal region inside $\Omega$ with one side lying on the displacement-free part $\Sigma$ of $\pa\Omega$ and  one side agrees with the fault segment $S$. We then align the computation grid with $S$ (see Figure \ref{fig:mesh}). 

\begin{figure}[h!]
  \centering
  \begin{subfigure}[b]{0.4\textwidth}
  \includegraphics[width=0.7\textwidth]{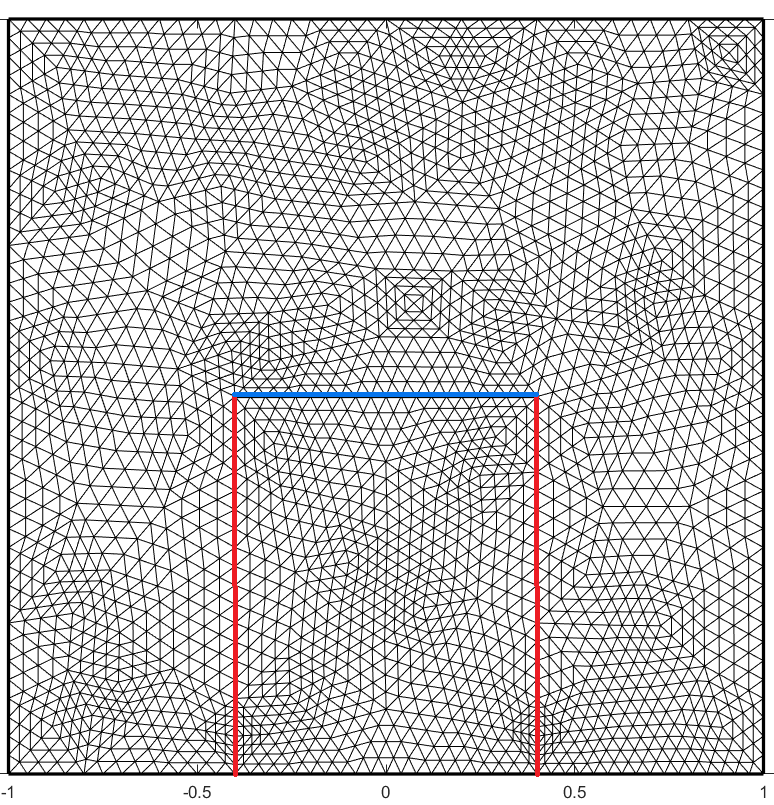}
  \caption{An example of the finer mesh used to solve the forward problem \eqref{eq:ForwardProblem} and generate the synthetic measurements, featuring the chosen trapezoidal region highlighted in red, which contains the dislocation line in blue.}\label{fig:mesh}
\end{subfigure}
\hfill
 \begin{subfigure}[b]{0.4\textwidth}
  \includegraphics[width=0.75\textwidth]{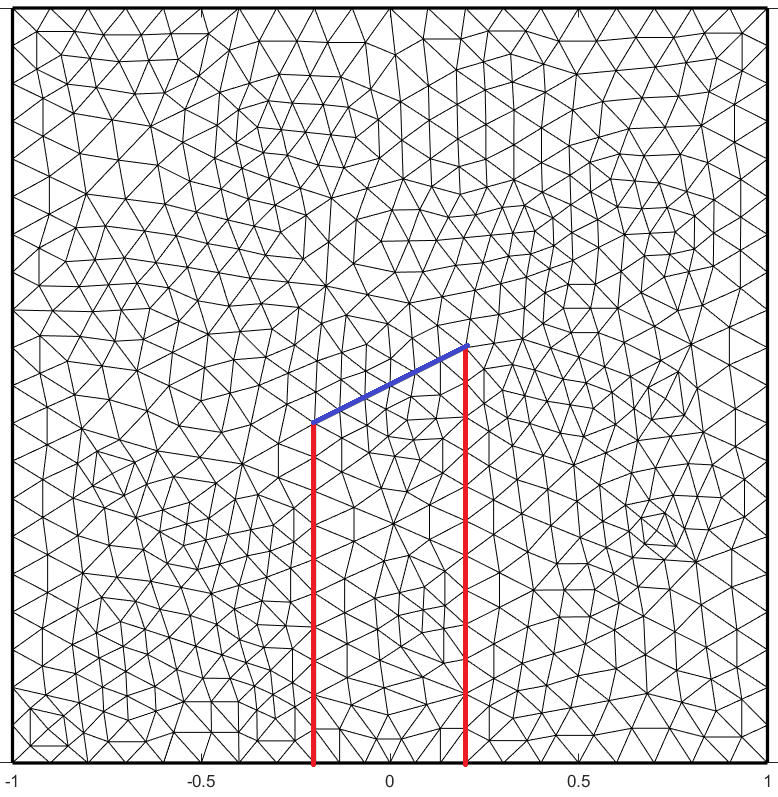}
  \caption{An example of the coarser mesh used to solve the inverse problem. The chosen trapezoidal region, highlighted in red, contains the dislocation line, which is shown in blue.
}\label{fig:mesh_ip}
\end{subfigure}
\caption{An example of the meshes employed for generating measurements and for solving the optimization problem.}
\label{fig:meshes}    
\end{figure}
We define $ F_I $ as the set of all interior sides. An interior side $\gamma$ in $F_I$ is characterized by the existence of two adjacent elements $ K^+ $ and $ K^- $ in $ \mathcal{T}_h $ such that $\gamma = \partial K^+ \cap \partial K^-$. For clarity in presenting the DG formulation, we distinguish between the sides that belongs to $S$ from the others,  which belong to $\mathcal{F}_I$.

Similarly, we label the sides of the elements that touch the boundary $\pa\Omega$ according to whether Dirichlet and Neumann conditions are imposed, grouping them into the families $ F_D $ and $ F_N $, respectively:
\begin{equation*}
F_D = \left\{ \gamma \mid \exists K \text{ such that } \gamma \subseteq \partial K \cap \partial \Omega_D \right\}
\end{equation*}
\begin{equation*}
F_N = \left\{ \gamma \mid \exists K \text{ such that } \gamma \subseteq \partial K \cap \partial \Omega_N \right\}
\end{equation*}
This definition assumes that the mesh conforms to the boundary decomposition of $ \partial \Omega $; that is, every side $\gamma$ belongs exclusively to either $ F_D $ or $ F_N $.\\

Next, given an integer $r\geq 1$ we define the finite-dimensional space of DG elements by:
\begin{equation*}
V_h^r = \left\{ \bv_h \in L^2(\Omega) \mid \bv^j_h \vert_K \in \mathcal{P}^r(K), \; j=1,2, \text{ for all } K \in \mathcal{T}_h \right\}
\end{equation*}
where $ \mathcal{P}^r(K) $ denotes the space of polynomials of total degree $ r $ defined on the element $ K $ and $\bv^j$ denotes the $j$-th component of the vector field $\bv$. 
The approximate solution will be sought in the space $V_{h}^r$, which is also the space of test functions.

\subsubsection*{Trace operators}
To handle piecewise discontinuous functions, we use appropriate trace operators. Let $ \gamma \in \mathcal{F}_I $ be an interior side shared by two elements $ K^+ $ and $ K^- $ of $ \mathcal{T}_h $, with unit normal vectors $ \bn^+ $ and $ \bn^- $ pointing outward from $ K^+ $ and $ K^- $, respectively. On $ \gamma $, we define the average and jump operators for sufficiently regular vector-valued and tensor-valued functions $ \bv $ and $ \boldsymbol{\upsigma}$ as follows:
\begin{equation}\label{eq:trace_operators}
\begin{aligned}
\{\bv\} &= \frac{1}{2} (\bv^+ + \bv^-),\ \qquad\qquad \{\boldsymbol{\upsigma}\} = \frac{1}{2} (\boldsymbol{\upsigma}^+ + \boldsymbol{\upsigma}^-)\\
\textrm{\textlbrackdbl} \bv \textrm{\textrbrackdbl} &= \bv^+ \otimes \bn^+ + \bv^- \otimes \bn^-,\quad
\textrm{\textlbrackdbl} \boldsymbol{\upsigma} \textrm{\textrbrackdbl} = \boldsymbol{\upsigma}^+ \bn^+ + \boldsymbol{\upsigma}^- \bn^-
\end{aligned}
\end{equation}
where $ \bv \otimes \bn = \frac{1}{2} (\bv \bn^T + \bn \bv^T) $. \\
The jump operator, defined using normal unit vectors, transforms a vector-valued function into a tensor-valued function and vice versa. This approach is preferable in numerical schemes compared to the common definition $ \textrm{\textlbrackdbl} \bv \textrm{\textrbrackdbl} = \bv^+ - \bv^- $, as it uses a sum instead of a difference, leveraging its commutative property. This ensures that the operator is independent of the element numbering $ (K^+ $ and $ K^-) $ during implementation.
For edges $ \gamma $ on the boundary of the domain, i.e., $ \gamma \in \mathcal{F}_D \cup \mathcal{F}_N $, we set:
\begin{equation*}
\begin{aligned}
\{\bv\} &= \bv^+,\ \qquad
\textrm{\textlbrackdbl} \bv \textrm{\textrbrackdbl} = \bv^+ \otimes \bn\\
\{\boldsymbol{\upsigma}\} &= \boldsymbol{\upsigma}^+,\qquad
\textrm{\textlbrackdbl} \boldsymbol{\upsigma} \textrm{\textrbrackdbl} = \boldsymbol{\upsigma}^+ \bn\\
\end{aligned}
\end{equation*}

\subsubsection*{Discontinuous Galerkin formulation}
In this section, we assume that the solution of \eqref{eq:ForwardProblem} is sufficiently regular so as to justify integration by parts.
Obtaining the weak formulation in the discontinuous Galerkin form is possible, in principle, even for solutions and test functions belonging to \( H^1(\Omega) \), provided that the traces of the gradient on the boundary of the elements are considered in appropriate distribution spaces. However, to achieve global a priori error estimates in the energy norm, higher regularity of the solution, such as assuming the solution in $H^2(\Omega)$,  is needed \cite{RivShaWhe03,SchXuZho06}, similarly to what happens in classical finite element methods. By construction, both the approximate solution as well as the test function belong to the space $V_h^r$, and hence have the required regularity.

To derive a variational formulation for the DG approximate solution, we begin by multiplying the first equation in \eqref{eq:ForwardProblem} by a test function $ \bv \in V_h^r $. For notational convenience, we extend $\bu$ and $\bv$ by zero to $\overline{\Omega}^c$, so that on $\pa\Omega$ the jump of $\bu$ and $\bv$ agrees with their trace value.
We then integrate over a generic element $ K \in \mathcal{T}_h $ and apply integration by parts. By summing over all elements $ K \in \mathcal{T}_h $, we obtain:
\begin{equation}\label{eq:DG_variational_formulation}
  \sum_{K \in \mathcal{T}_h} \int_K \CC \nablahat \bu : \nablahat \bv \, dx 
 = \sum_{K \in \mathcal{T}_h} \int_{\partial K} (\CC \nablahat \bu \bm{n}) \cdot \bv \, dx.
\end{equation}
We note that on the right-hand side of the identity above,  the integrals over the sides of triangles in the partition are computed twice for triangles sharing a common side, once with the outward normal vector $ \bn^+ $ and once with $ \bn^- $.

Since $\bu$ solves \eqref{eq:ForwardProblem}, the traction is zero on all sides in $\cF_N$, so that we can rewrite \eqref{eq:DG_variational_formulation} as follows:
\begin{equation*}
  \sum_{K \in \mathcal{T}_h} \int_{\partial K} (\CC \nablahat \bu \bn) \cdot \bv \, dx 
= \sum_{\gamma \in \mathcal{F}_I\cup S} \int_{\gamma} \left[ (\CC\nablahat \bu)^+ \bn^+ \cdot \bv^+ + (\CC\nablahat \bu)^- \bn^- \cdot \bv^-) \right] d\sigma(\bx) 
+ \sum_{\gamma \in \mathcal{F}_D} \int_{\gamma} (\CC\nablahat \bu \bn) \cdot \bv \, d\sigma(\bx),
\end{equation*}
Next, from the definitions of the average and jump operators given in \eqref{eq:trace_operators} we have that 
\begin{equation}\label{eq:Integral_bound_DG}
\begin{aligned}
\sum_{\gamma \in \mathcal{F}_I\cup S} \int_{\gamma} \left[ (\CC \nablahat \bu)^+ \bn^+ \cdot \bv^+ +(\CC \nablahat \bu)^- \bn^- \cdot \bv^-\right] d\sigma(\bx) 
+ \sum_{\gamma \in \mathcal{F}_D} \int_{\gamma} (\CC\nablahat \bu \bn) \cdot \bv \, d\sigma(\bx)\\
= \sum_{\gamma \in \mathcal{F}_I \cup \mathcal{F}_D\cup S} \int_{\gamma} \{\CC\nablahat\bu\} : \textrm{\textlbrackdbl} \bv \textrm{\textrbrackdbl}\, d\sigma(\bx) 
+ \sum_{\gamma \in \mathcal{F}_I\cup S} \int_{\gamma}  \textrm{\textlbrackdbl}\CC\nablahat \bu  \textrm{\textrbrackdbl} \cdot \{\bv\}\, d\sigma(\bx).
\end{aligned}
\end{equation}
The second integral term on the right-hand side vanishes as $ \textrm{\textlbrackdbl}\CC\nablahat \bu \textrm{\textrbrackdbl}=0$ on every interior side of the partition and across $S$, again using that $\bu$ is a variational solution of \eqref{eq:ForwardProblem}. The first integral term on the right-hand side does not vanish in general, because the test functions in $V^r_h$ may jump across the sides of the partition. Since the bilinear form on the left-hand side of \eqref{eq:DG_variational_formulation} is symmetric, for numerical accuracy it is important to preserve symmetry after integration by parts. To this end, using that the jump $\textrm{\textlbrackdbl}\bu\textrm{\textrbrackdbl}=0$ on $\cF_I\cup \cF_D$, $\textrm{\textlbrackdbl}\bu\textrm{\textrbrackdbl}=\bg$ on $S$,
and inserting \eqref{eq:Integral_bound_DG} for the right-hand side of 
\eqref{eq:DG_variational_formulation},
we obtain the following variational formulation, which must be satisfied by the DG approximate solution:
\begin{equation}\label{eq:DG_aux2}
\begin{aligned}
    \sum_{K \in \mathcal{T}_h} \int_K \CC \nablahat \bu : \nablahat \bv \, dx &- \sum_{\gamma \in \mathcal{F}_I \cup \mathcal{F}_D\cup S} \int_{\gamma} \{\CC\nablahat\bu\} : \textrm{\textlbrackdbl} \bv \textrm{\textrbrackdbl}  d\sigma(\bx) -\sum_{\gamma \in \mathcal{F}_I \cup \mathcal{F}_D\cup S} \int_{\gamma} \{\CC\nablahat\bv\} : \textrm{\textlbrackdbl} \bu \textrm{\textrbrackdbl}  d\sigma(\bx)\\
    &=-\sum_{\gamma\in S}\int_{\gamma}\bg\otimes \bn^+ : \{\CC\nablahat \bv\}\, d\sigma(\bx) 
\end{aligned}
\end{equation}
It has been shown in \cite{AntMazQuaRap12,AntAyuMazQua16} that adding a penalization of jumps in the form
\begin{equation*}
\int_{\gamma} \eta \textrm{\textlbrackdbl} \bu \textrm{\textrbrackdbl} : \textrm{\textlbrackdbl} \bv \textrm{\textrbrackdbl} d\sigma(\bx)
\end{equation*}
ensures the stability of the numerical method while preserving symmetry. It is crucial to choose the penalization parameter appropriately.
We define 
$\eta \in L^1(\mathcal{F}_I \cup \mathcal{F}_D\cup S)$ as follows:
\begin{equation*}
\eta_{\gamma} = \beta \mathfrak{C} \frac{r^2}{h},\ \qquad \forall \gamma \in \mathcal{F}_I \cup \mathcal{F}_D\cup S, 
\end{equation*}
where $\beta$ is a positive constant, $h$ is the mesh size, $r$ is the polynomial degree, and $\mathfrak{C} (\bx) := \left( \sum_{i,j,h,k} |\CC_{ijhk}(\bx)|^2 \right)^{1/2}$ is the discrete $2$-matrix norm .
As the mesh is refined, specifically as $h \to 0$, the parameter $\eta$ increases, thereby providing greater stabilization. 
The penalized DG variational formulation is finally given by:
\begin{equation}
\begin{aligned}
    &\sum_{K \in \mathcal{T}_h} \int_K \CC \nablahat \bu : \nablahat \bv \, dx - \sum_{\gamma \in \mathcal{F}_I \cup \mathcal{F}_D\cup S} \int_{\gamma} \{\CC\nablahat\bu\} : \textrm{\textlbrackdbl} \bv \textrm{\textrbrackdbl}  d\sigma(\bx)\\
    &-\sum_{\gamma \in \mathcal{F}_I \cup \mathcal{F}_D\cup S} \int_{\gamma} \{\CC\nablahat\bv\} : \textrm{\textlbrackdbl} \bu \textrm{\textrbrackdbl}  d\sigma(\bx)+\sum_{\gamma\in \mathcal{F}_I\cup \mathcal{F}_D\cup S}\int_{\gamma} \eta \textrm{\textlbrackdbl} \bu \textrm{\textrbrackdbl} : \textrm{\textlbrackdbl} \bv \textrm{\textrbrackdbl} d\sigma(\bx)\\
    &=-\sum_{\gamma\in S}\int_{\gamma}\bg\otimes \bn^+ : \{\CC\nablahat \bv\}\, d\sigma(\bx)+\sum_{\gamma\in S}\int_{\gamma} \eta \bg\otimes \bn^+ : \textrm{\textlbrackdbl} \bv \textrm{\textrbrackdbl} d\sigma(\bx).
\end{aligned} \label{eq:DGFormPenalized}
\end{equation}
Then the DG approximate solution $\bu_h\in V^r_h$ of the forward problem \eqref{eq:ForwardProblem}  must satisfy \eqref{eq:DGFormPenalized} for all $v\in V^r_h$, see \cite{RivShaWhe03,SchXuZho06,ArnBreCocMar01,AntMazQuaRap12,AntAyuMazQua16}. 
We use a similar penalized formulation to compute the DG solution $\bw_h\in V^r_h$ of the adjoint problem \eqref{eq:AdjointProblem}. The numerical implementation then follows the standard approach for finite elements methods.

In Figure \ref{fig:sol_forward_pb}, we plot an example of the numerical solution for the forward problem \eqref{eq:ForwardProblem},  obtained by using linear elements, that is, polynomials of degree $1$ on each element $K$, and setting the mesh size $h=0.01$. We consider the case of a horizontal dislocation  with vertices at $(-0.4,0)$ and $(0.4,0)$ and a slip field vanishing at the endpoint of the dislocation segment given by 
\begin{equation}
\bg(s) = \left( 
10 \, \left( -\left(\frac{100}{39}\right)^{12} \, s^{12} + 1 \right) \,\chi_{(-\frac{39}{100}, \frac{39}{100})}(s),
\,20 \, \left( -\left(\frac{100}{39}\right)^{12} \, s^{12} + 1 \right) \, \chi_{(-\frac{39}{100}, \frac{39}{100})}(s) 
\right),
\end{equation}

\begin{figure}[H]
  \centering
  \includegraphics[width=\textwidth]{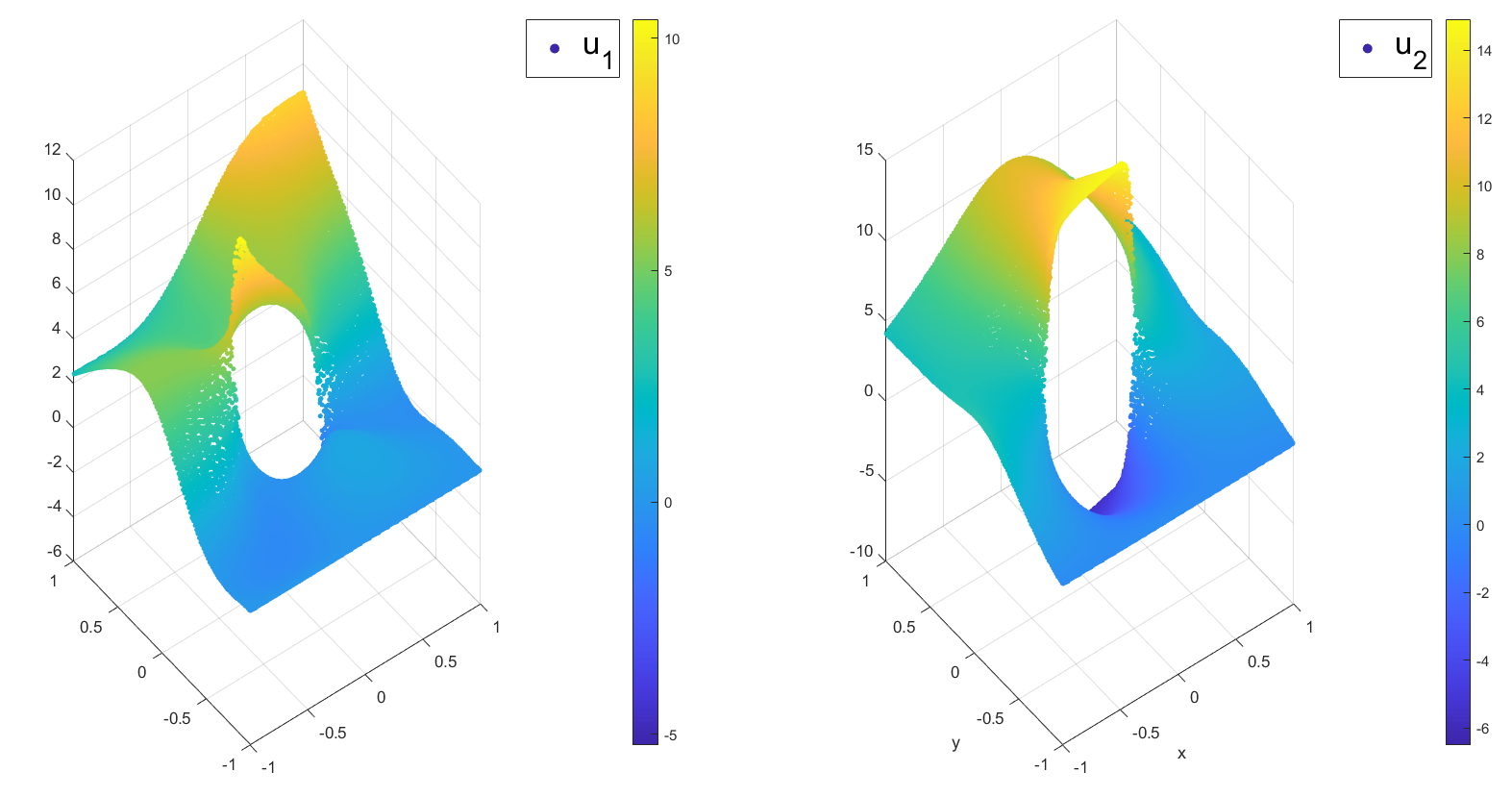}
  \caption{An example of the solution to the forward problem \eqref{eq:ForwardProblem} obtained using the Discontinuous Galerkin method.}
  \label{fig:sol_forward_pb}
\end{figure}

\subsection{Reconstruction Algorithm}
For the reconstruction algorithm, we adopt the gradient descent approach introduced in \cite{BMPS18}. Below, we briefly recall the method as it applies to our problem.

To highlight the dependence of the shape derivative on the map $\mathcal{U}\in W^{1,\infty}(\Omega)$, identified with a vector field on $\RR^2$, that encodes the infinitesimal movement of the fault (see \eqref{eq:FaultDeformationDef}), in this section we will use the notation 
\begin{equation*}
     \dfrac{d \cJ}{d\tau}\lfloor_{\tau=0}=\left\langle D\cJ(S,\bg),\cU\right \rangle,
\end{equation*}
where $D\cJ(S,\bg)$ represents the gradient.  
Furthermore, in all simulations we will assume the slip field $\bg$ is known at the dislocation, focusing solely on reconstructing the dislocation line.
As we have already mentioned, if $S$ can be reconstructed, the reconstruction of $\bg$ is conceptually simpler as this part of the inverse problem is linear.
In what follows, we will focus on presenting the descent algorithm for the reconstruction of $S$ given $\bg$.
Consequently, for notational convenience we simply write
\begin{equation*}
\dfrac{d \cJ}{d\tau}\lfloor_{\tau=0}=\left\langle D\cJ(S),\cU\right \rangle.
\end{equation*}
We also use $((,))$ to denote the $L^2$-inner product.

We also continue to specialize to the case that the dislocation line $S$ is a straight segment, so it is uniquely determined by its two vertices $\bm{P}_l\in \RR^2$, $l=1,2$.

A key step in the descent algorithm is the computation of the descent direction at each step. Let $X \subset W^{1,\infty}(\Omega, \mathbb{R}^2)$ be a suitable subspace. Then, at iteration level $k\in \NN$, we solve the equation
\begin{equation}\label{eq:descent_alg_aux1}
(( \btheta^k,\delta\btheta)) + \langle D\cJ(S^k), \delta \btheta\rangle = 0
\end{equation}
for every $\delta \btheta \in X$, where $\btheta^k \in X$ denotes the descent direction and $S^k$ is the dislocation line computed at the $k$-th iteration. The subspace $X$ is chosen so that $S^k$ is a segment contained in $\Omega$ for each $k$.
For solving \eqref{eq:descent_alg_aux1}, we discretize the problem  and employ the idea already contained in \cite{BMPS18}.
That is, given that the dislocation line is uniquely defined by its vertices, we compute the descent direction for each vertex individually before updating the partition at each iteration $k$. 

Specifically, we recall that we have two meshes. A coarse mesh used for the measurements and a fine mesh used for the DG approximation.
The vertex of the dislocation segment are assumed to be on both meshes. At iteration level $k$, for each vertex $\bm{P}_l^k$, $l=1,2$, of $S^k$, we calculate the gradient of the functional $\mathcal{J}$ with respect to the position of the vertex $\bmP^k_l$ as follows. We select two maps $\mathcal{U}^k_{l,1}$, $\mathcal{U}^k_{l,2}\in W^{1,\infty}(\Omega, \mathbb{R}^2)$ with the properties that:
\begin{enumerate}
\item $\mathcal{U}^k_{l,1}, \mathcal{U}^k_{l,2}$ have support strictly contained within $\Omega$.
\item $\mathcal{U}^k_{l,1}$ and $\mathcal{U}^k_{l,2}$ are piecewise linear along the sides of the (coarse) partition and satisfy
\begin{equation*}
\mathcal{U}^k_{l,1}(\bmP^k_j) = (\delta_{jl}, 0), \quad \mathcal{U}^k_{l,2}(\bmP^k_j) = (0, \delta_{jl}),
\end{equation*}
where $\delta_{jl}$ denotes the Kronecker delta and $\bmP^k_j$ denotes any vertex of the mesh.
\end{enumerate}
Then, the gradient of $\mathcal{J}$ with respect to the position of the vertex $\bmP^k_l$ is given by:
\begin{equation*}
         \left(\Big\langle  D\cJ(S^k), \mathcal{U}^k_{l,1}\Big\rangle,\Big\langle  D\cJ(S^k), \mathcal{U}^k_{l,2}\Big\rangle\right),
\end{equation*}
and hence the steepest descent at the $k$-th iteration is represented by the vector
\begin{equation}\label{eq:descent_direction}
\btheta^k_{l,j} = -\left(\Big\langle  D\cJ(S^k), \mathcal{U}^k_{l,1}\Big\rangle,\Big\langle  D\cJ(S^k), \mathcal{U}^k_{l,2}\Big\rangle\right),
\end{equation}

For the DG approximation, we define the maps $\mathcal{U}^k_{l,1}$ and $\mathcal{U}^k_{l,2}$ as follows:
\begin{equation*}
\mathcal{U}^k_{l,1} = (\varphi^k_l, 0), \quad \mathcal{U}^k_{l,2} = (0, \varphi^k_l),
\end{equation*}
where $\varphi^k_l$ is the hat function associated with the node $\bmP^k_l$ of the coarse mesh (see, for example, Figure \ref{fig:mesh_ip}); that is, $\varphi^k_l$ is piecewise linear and 
\begin{equation*}
\varphi^k_l(\bmP^k_j) = \delta_{jl}.
\end{equation*}
We then stop the algorithm when a desired tolerance is reached or a maximum number of iterations has occurred as customary.
Below is a pseudocode for the reconstruction algorithm as implemented in MATLAB$^{\text{\copyright}}$ for a specific set of parameters.

\bigskip

\begin{algorithm}
\caption{Reconstruction Algorithm}
\begin{algorithmic}[1]
\State \textbf{Parameters:} $maxIter = 5000$, $tol = 1 \times 10^{-7}$, $\alpha = 1 \times 10^{-6}$, $nIter=150$
\State Initialize geometry and mesh
\For{$k = 1$ \textbf{to} $maxIter$}
    \State Solve the forward problem \eqref{eq:ForwardProblem}
    \State Compute gradient of the solution. 
    \State Extract boundary data for adjoint problem
    \State Solve the adjoint problem \eqref{eq:AdjointProblem}
    \State Compute gradient of the adjoint solution. 
    \State Calculate shape derivatives using $\mathcal{U}^k_1$ and $\mathcal{U}^k_2$. 
    \State Update the vertices of the dislocation, using \eqref{eq:descent_direction}, i.e. $\bm{P}^{k+1}_l=\bm{P}^k_l-\alpha \theta^k_l$.
    \If{$k > nIter$ \textbf{and}  max($\theta^{k}_l$) $<$ $tol$}
        \State \textbf{break} \text{ (Convergence reached)}
    \EndIf
   \State Update geometry and mesh.
\EndFor
\end{algorithmic}
\end{algorithm}

\subsubsection{Numerical tests}

All the numerical simulations presented in this section has been performed in MATLAB$^{\text{\copyright}}$.
As already mentioned, the computational domain is a square $\Omega = [-1,1] \times [-1,1]$, with homogeneous Dirichlet boundary conditions on $y = -1$, modeling the buried part of the boundary,  and homogeneous Neumann boundary conditions on the other three sides, modeling the exposed part of the boundary at the Earth's surface (see Figure \ref{f:NumericalSetUp}).

We include noise to the data to model measurement errors. Specifically, for a noiseless boundary measurement $ \bu_m \in H^{1/2}(\partial \Omega) $, the noisy data $ \widehat{\bu}_m $ is generated by adding uniform noise to $ \bu_m $ as follows:
\begin{equation*}
\widehat{\bu}_m(\bx) = \bu_m(\bx) + \varepsilon \|\bu_m\|_{L^2(\partial\Omega)}.
\end{equation*}
where $ \varepsilon $ is a uniform random variable drawn from the interval $ (-a, a) $, with $ a > 0 $ determined based on the desired noise level (unless specified otherwise, we set $a=0.0007$). In the numerical implementation, the noise is added pointwise to each node of the coarse mesh used for the measurements.
To quantify the noise level, we compute the relative $L^2$ error:
\begin{equation*}
\frac{\| \widehat{\bu}_m - \bu_m \|_{L^2(\partial \Omega)}^2}{\| \bu_m \|_{L^2(\partial \Omega)}^2}.
\end{equation*}
We include this error for each of the numerical tests in the corresponding figure below.

We perform simulations both in the case of a slip field $\bg\in H^\half_{00}(S)$, for which the variational formulation of the forward problem and the shape derivative is rigorously justified, as well as in the case of constant slip field $\bg$, a configuration often used in seismology. For the constant slip field, we study the case of a highly anisotropic oblique slip, setting:
\begin{equation*}
\bg = (-10, 20).
\end{equation*}
The presence of singularities in the elastic displacement at the tips of the dislocation aides in the reconstruction.

For the case of a slip vanishing at the endpoint of the dislocation segment, we choose a slip that vanishes to high order, which makes the reconstruction more difficult:
\begin{equation}\label{eq:numerical_g_compact_supp}
\bg(s) = \left( 
-10 \, \left( -\left(\frac{100}{39}\right)^{12} \, s^{12} + 1 \right) \,\chi_{(-\frac{39}{100}, \frac{39}{100})}(s),
\,20 \, \left( -\left(\frac{100}{39}\right)^{12} \, s^{12} + 1 \right) \, \chi_{(-\frac{39}{100}, \frac{39}{100})}(s) 
\right),
\end{equation}
where $\chi_{(-\frac{39}{100}, \frac{39}{100})}(s)$ is the characteristic function of the interval $[-39/100,39/100]$, and the scalar $s$ parameterizes the segment.
Both cases are designed to test the limitation of the algorithm.

Similarly, the regularization parameter $\alpha$ used for updating the dislocation vertices is chosen sufficiently small, but non zero,  $\alpha = 10^{-6}$. Finally, the elastic coefficients are taken constant to simplify the expression of the shape derivative and to improve 
stability. For simplicity, we set the Lam\'e coefficients $\lambda = \mu = 1$.

For most of the numerical tests,  we  measure the displacement on the entire exposed portion of $\pa\Omega$, the top and the two lateral sides of the square $\Omega$. One measurement is enough in this case to reconstruct the dislocation line $S$. For comparison, we include a test for which the displacement is measured only on the top side of  $\Omega$ (see Figure \ref{fig:one_meas} and Test 5). As expected, the reconstruction with one measurement is poor, and two measurements are needed in this case.

\bigskip

We include the results for the following numerical tests:
\bigskip

\textbf{Test 1: Horizontal dislocation with constant slip field $\bg$ (Figure \ref{fig:reconstruction1}).}
\begin{figure}[H]
  \centering
  \includegraphics[width=0.8\textwidth]{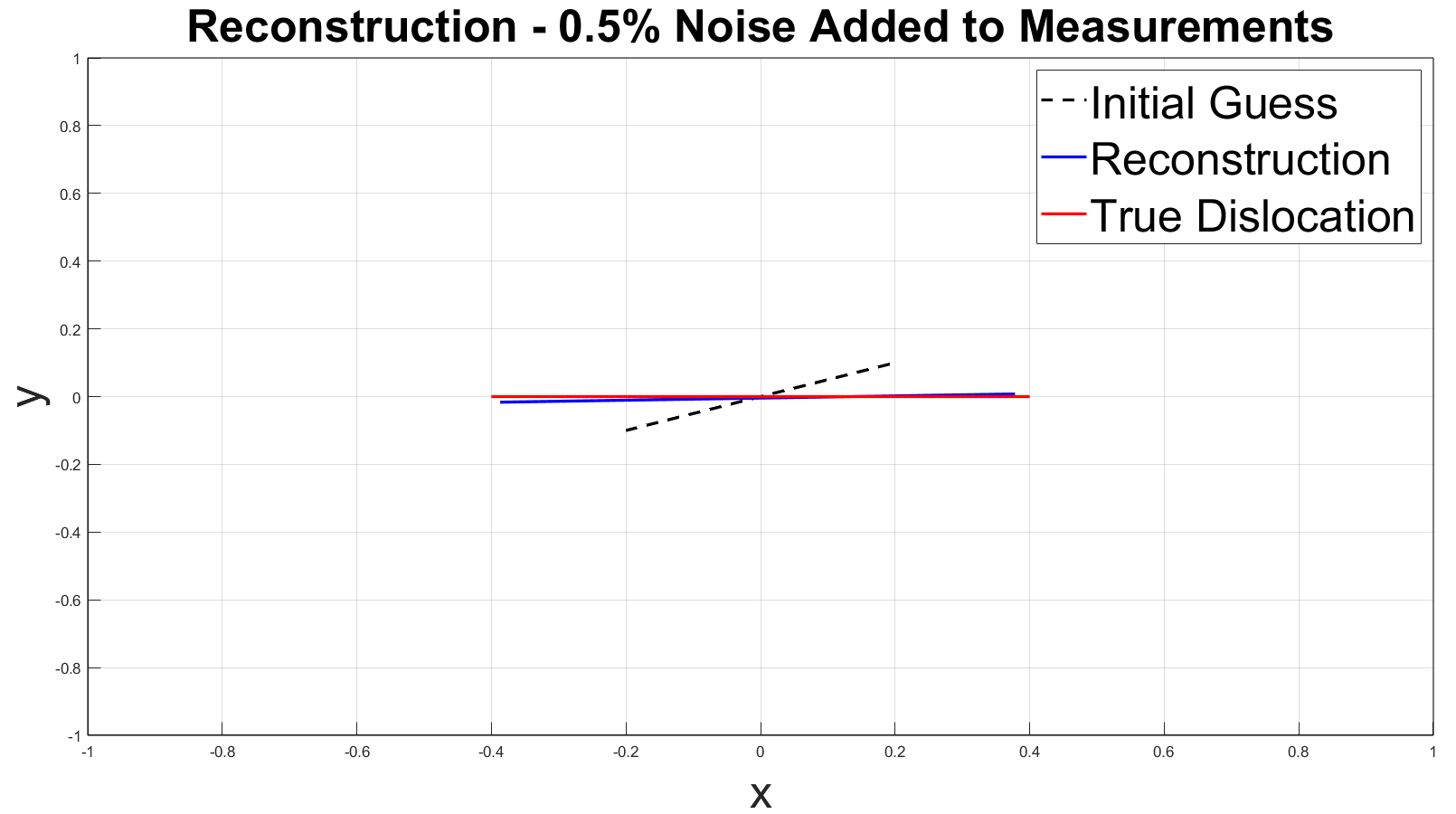}
  \caption{Reconstruction of a horizontal dislocation  with vertices at $(-0.4,0)$ and $(0.4,0)$, using only one boundary measurement with $0.5\%$  noise level . The initial guess for the reconstruction algorithm is represented by  the green dashed line. The slip-field along the dislocation line is constant, $\bg=(-10,20)$.}
  \label{fig:reconstruction1}
\end{figure}

\bigskip

\textbf{Test 2: Horizontal dislocation with  slip field $\bg$ vanishing at the endpoints (Figure \ref{fig:reconstruction2})}
\begin{figure}[H]
  \centering
  \includegraphics[width=0.8\textwidth]{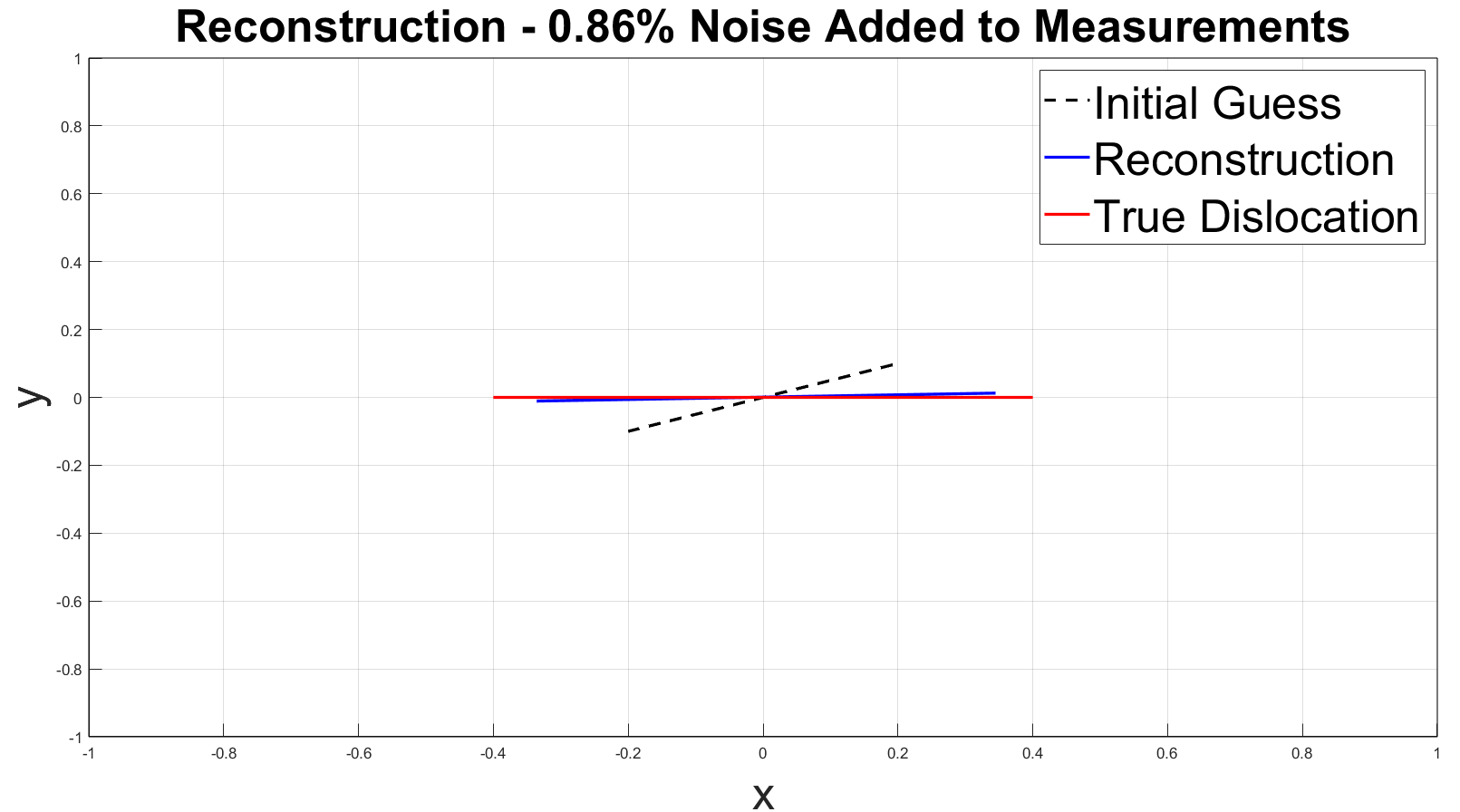}
  \caption{Reconstruction of a horizontal dislocation with vertices at  $(-0.4,0)$ and $(0.4,0)$, using only one boundary measurement with  $0.86\%$ noise level . The initial guess for the reconstruction algorithm is represented by the green dashed line. The slip field $\bg$ is chosen as in \eqref{eq:numerical_g_compact_supp}.}
  \label{fig:reconstruction2}
\end{figure}

\bigskip

\textbf{Test 3: Oblique dislocation with constant slip field $\bg$ (Figure \ref{fig:reconstruction3})}
\begin{figure}[H]
  \centering
  \includegraphics[width=0.8\textwidth]{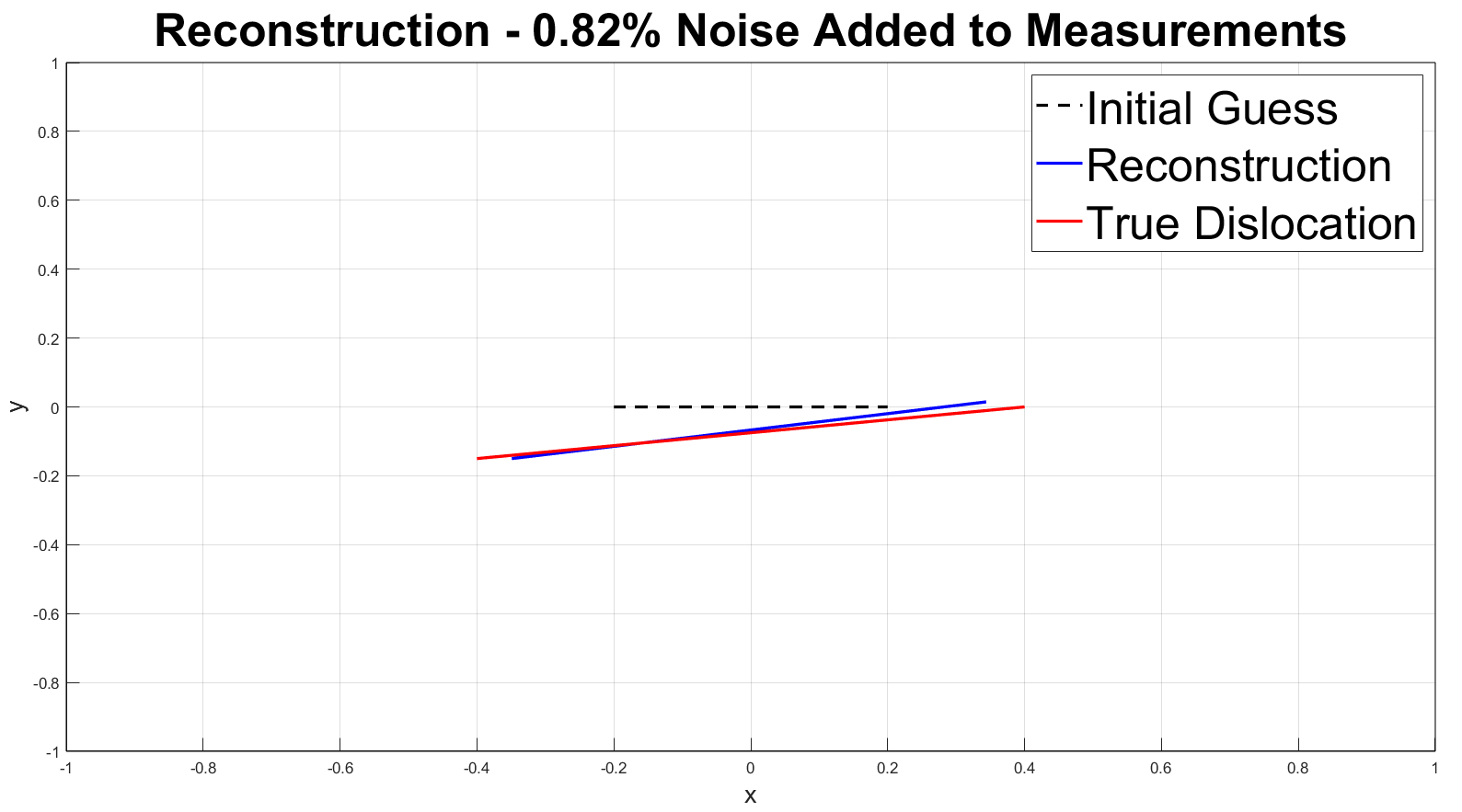}
  \caption{Reconstruction of an oblique dislocation with vertices at $(-0.4,-0.15)$ and $(0.4,0)$, using only one boundary measurement with $0.82\%$ noise level. The initial guess for the reconstruction algorithm is represented by the green dashed line. The slip field along the dislocation line is constant, $\bg=(-10,20)$.}
  \label{fig:reconstruction3}
\end{figure}

\bigskip

\textbf{Test 4: Oblique dislocation with slip field $\bg$  vanishing at the endpoints (Figure \ref{fig:reconstruction4})}
\begin{figure}[H]
  \centering
  \includegraphics[width=0.8\textwidth]{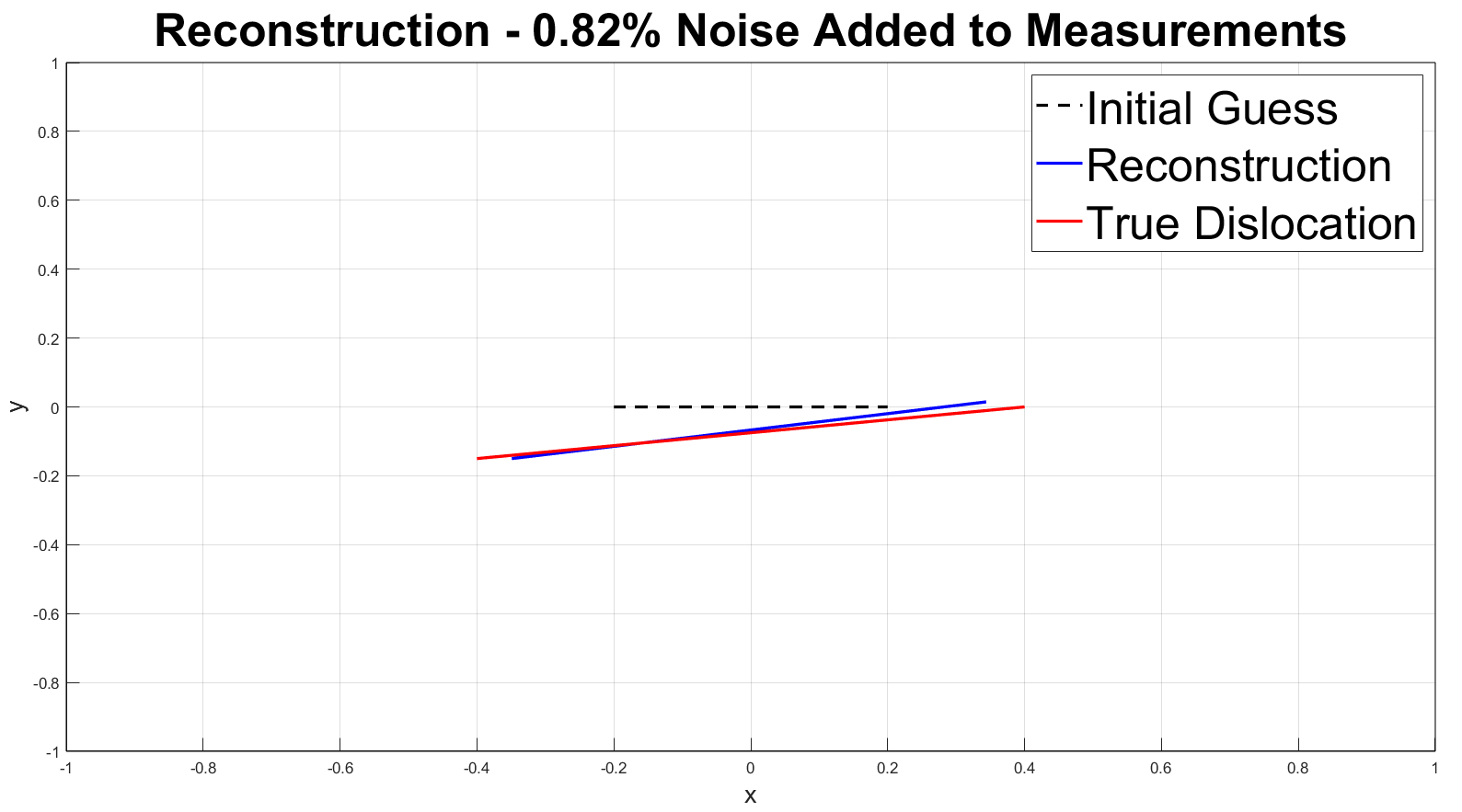}
  \caption{Reconstruction of an oblique dislocation with vertices at $(-0.4,-0.15)$ and $(0.4,0)$, using only one boundary measurement with $0.82\%$ noise level.   The initial guess for the reconstruction algorithm is represented by the green dashed line. The slip field $\bg$ is chosen as in \eqref{eq:numerical_g_compact_supp}.} \label{fig:reconstruction4}
\end{figure}

\bigskip

\textbf{Test 5: Horizontal dislocation with slip field $\bg$  vanishing at the endpoints and measurements only on one side  (Figure \ref{fig:one_meas})}
In this numerical test, we consider the case where measurements can only be collected on one side, specifically along the side given by the points $\{y=1\}$. As expected, in this case, the reconstructions are less accurate. In particular, when assuming knowledge of only one boundary displacement measurement, the reconstructions are of  poor quality (see Figure \ref{fig:reconstruction5_a}). However, when more measurements are available (still on the same side), the situation gradually improves (see Figure \ref{fig:reconstruction5_b} for the case of two measurements), even if the results are still not satisfactory. In this latter case, to generate the second measurement, we use the slip field $\bg$ defined by the following data
\begin{equation*}
\bg(s) = \left( 
60 \, \left( -\left(\frac{1000}{395}\right)^{10} \, s^{10} + 1 \right) \,\chi_{(-\frac{395}{1000}, \frac{395}{1000})}(s),
\,-10 \, \left( -\left(\frac{1000}{395}\right)^{10} \, s^{10} + 1 \right) \, \chi_{(-\frac{395}{1000}, \frac{395}{1000})}(s) 
\right),
\end{equation*}

\begin{figure}[H]
  \centering
  \begin{subfigure}[b]{0.6\textwidth}
\includegraphics[width=\textwidth]{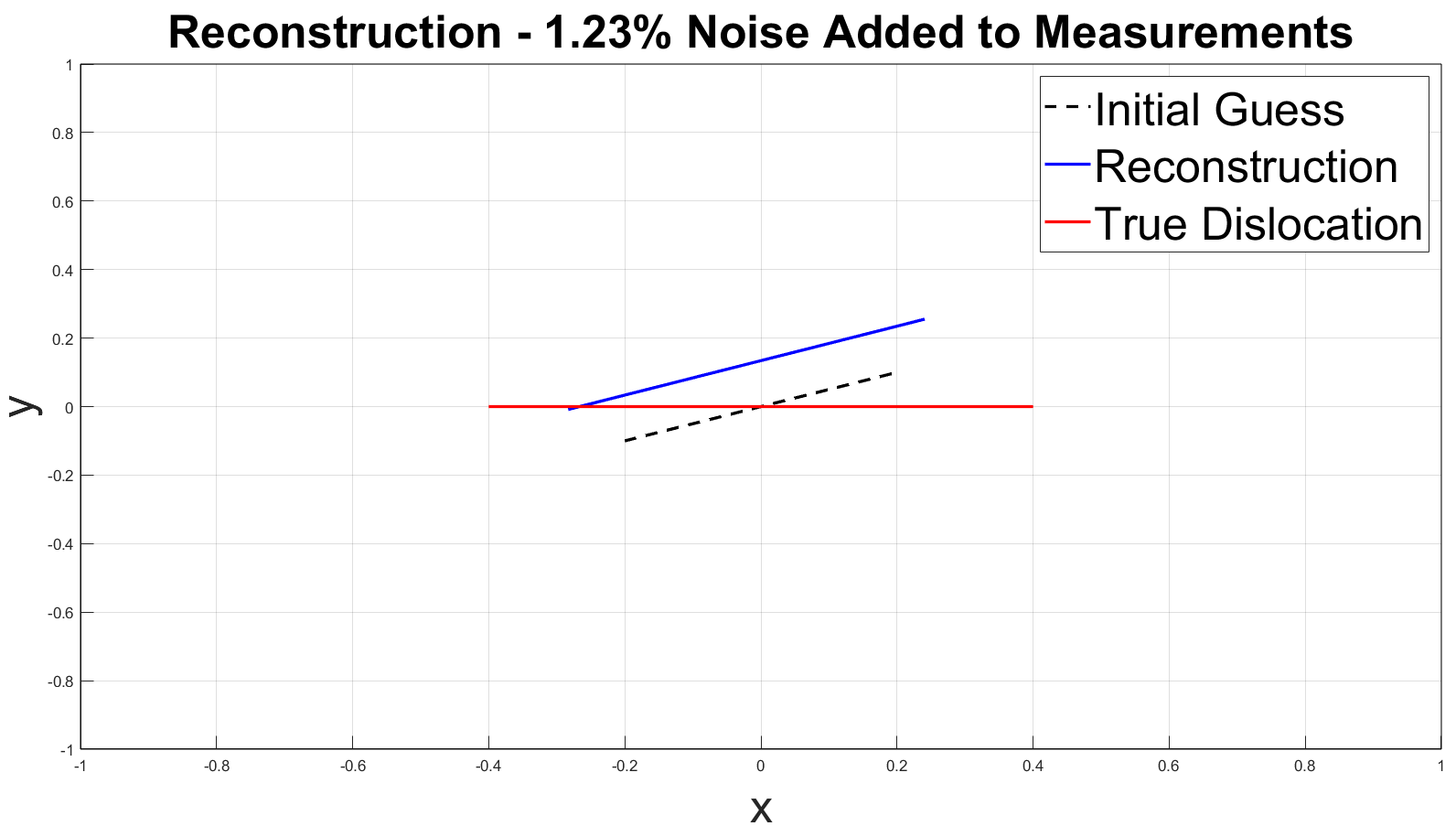}
 \caption{Reconstruction of a horizontal dislocation between vertices $(-0.4,0)$ and $(0.4,0)$, using only one boundary measurement with $1.23\%$ noise level. The initial guess for the reconstruction algorithm is represented by the green dashed line. The slip field $\bg$ is chosen as in \eqref{eq:numerical_g_compact_supp}.}\label{fig:reconstruction5_a}
\end{subfigure}
\hfill
 \begin{subfigure}[b]{0.6\textwidth}
  \centering
\includegraphics[width=\textwidth]{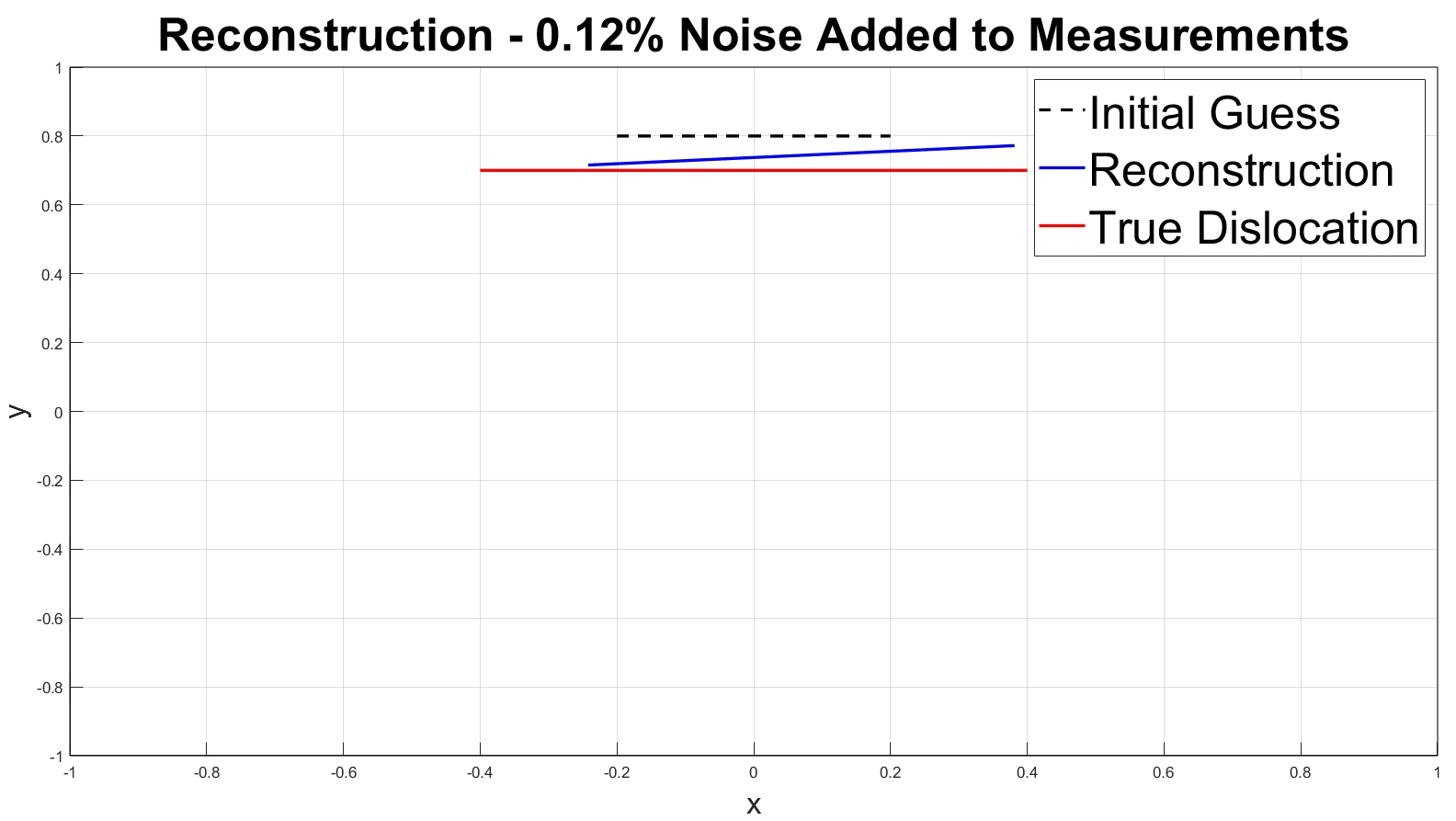}
 \caption{Reconstruction of a horizontal dislocation between vertices $(-0.4,0.7)$ and $(0.4,0.7)$, using two boundary measurements with $0.12\%$ noise level. The initial guess for the reconstruction algorithm is represented by the green dashed line. The slip field $\bg$ is chosen as in \eqref{eq:numerical_g_compact_supp}.}\label{fig:reconstruction5_b}
\end{subfigure}
\caption{The case of measurements only on the top side $y=1$ of the square $\Omega$.}
\label{fig:one_meas}    
\end{figure}

\section*{Acknowledgments}
The authors thank Paola Antonietti for useful discussion and for providing a MATLAB$^\text{\copyright}$ package to implement the DG method. A. Mazzucato was partially supported by the US National Science Foundation Grants DMS-1909103, DMS-2206453, and Simons Foundation Grant 1036502, and thanks the hospitality of the Mathematics Program at New York University-Abu Dhabi (NYUAD) and of the Mathematics Department at Milan University, where part of this work was conducted. Elena Beretta's research has been partially supported by NYUAD Science Program Project Fund AD364.  A. Lee was also partially supported by the US National Science Foundation Grants DMS-1909103 and DMS-2206453. A. Aspri is member of the group GNAMPA 
(Gruppo Nazionale per l’Analisi Matematica, la Probabilità e le loro Applicazioni) of INdAM (Istituto Nazionale di Alta Matematica).

\subsection*{Conflict of Interest Statement}
The authors declare that they have no conflict of interest.

\subsection*{Data Availability}
All the data used in this work is synthetic data generated by an algorithm coded in MATLAB$^{\text{\copyright}}$. A description of the algorithm is 
included in the body of the manuscript. Output of the numerical scheme can be made available upon request. 


\end{document}